\documentclass[a4paper, 11pt]{article}
\usepackage{amsmath,amssymb,esint,amscd,xspace,fancyhdr,color,authblk,srcltx,fontenc,bbm}
\usepackage{hyperref}
\usepackage{cite}
\newtheorem{theorem}{Theorem}[section]

\newtheorem{definition}[theorem]{Definition}
\newtheorem{lemma}[theorem]{Lemma}
\newtheorem{proposition}[theorem]{Proposition}
\newtheorem{remark}[theorem]{Remark}

\newenvironment{proof}[1][Proof]{\textbf{#1.} }{\hfill\rule{0.5em}{0.5em}}
{\catcode`\@=11\global\let\AddToReset=\@addtoreset
\AddToReset{equation}{section}

\AddToReset{theorem}{section}

\title{Calder\'on-Zygmund gradient estimates for $p$-Laplace systems with BMO complex coefficients}

\author{Van-Chuong Quach\thanks{Department of Mathematics, Dong Nai University, Bien Hoa City, Dong Nai Province, Vietnam; \texttt{chuongqv@dnpu.edu.vn}}; Thanh-Nhan Nguyen\thanks{Group of Analysis and Applied Mathematics, Department of Mathematics, Ho Chi Minh City University of Education, Ho Chi Minh city, Vietnam; \texttt{nhannt@hcmue.edu.vn}}; Minh-Phuong Tran\footnote{Corresponding author} \thanks{Applied Analysis Research Group, Faculty of Mathematics and Statistics, Ton Duc Thang University, Ho Chi Minh city, Vietnam; \texttt{tranminhphuong@tdtu.edu.vn}}}

\date{\today}

\begin{document}

\maketitle
\begin{abstract}

This work is concerned with global gradient bounds for a class of divergence-form degenerate elliptic systems with complex-valued coefficients. Notably, the leading coefficients are merely required to be sufficiently small in BMO, which is strictly weaker than the VMO condition. In the complex setting, the well-posedness of this problem was recently investigated in~\cite{KV25}, where the authors established a strong accretivity condition on the leading coefficients, and this structural condition allows them to derive Schauder-type estimates for weak solutions. In our study, it has already been observed that gaining existence and uniqueness of weak solutions is possible under a natural and less restrictive assumption on the complex-valued coefficients. Following this direction, we prove a global Cader\'on-Zygmund-type estimate for weak solutions, from which the Morrey-space regularity follows as a consequence. This paper is a contribution to the better understanding of solution behavior and may be viewed as part of a series of works aimed at extending regularity theory in the complex-valued setting.

\medskip


\medskip


\end{abstract}
\maketitle

\tableofcontents
\section{Introduction and main results}
\label{sec:intro}

In this paper, we are interested in the Cauchy-Dirichlet problem of degenerate elliptic systems
\begin{align}\label{eq-main}
-\mathrm{div}\left(\mathfrak{a}(x)\mathbb{V}_{p}^{\mu}(\mathrm{D}\mathbf{u})\right) = -\mathrm{div}\left(|\mathbf{F}|^{p-2}\mathbf{F}\right)  \mbox{ in } \Omega, \mbox { and } \mathbf{u} = \mathbf{0} \mbox{ on } \partial \Omega,
\end{align}
where $\Omega$ is an open bounded subset of $\mathbb{R}^n$, with $n \ge 2$; the exponent $p \in (1,\infty)$; the given source term $\mathbf{F}: \Omega \to \mathbb{C}^{N \times n}$ with $N \in \mathbb{N}$ arbitrary and $N \ge 1$; the given mapping $\mathfrak{a}: \Omega \to \mathbb{C}$ is bounded measurable complex-valued function defined on $\Omega$. In addition, the problem poses a structural operator associated with the growth behavior $\mathbb{V}_{p}^{\mu}: \mathbb{C}^{N\times n} \to \mathbb{C}^{N\times n}$ for $\mu \in [0,1]$, which is defined by 
\begin{align}\label{def-Vps}
\mathbb{V}_{p}^{\mu}(\eta) = (\mu^2 + |\eta|^{2})^{\frac{p-2}{2}}\eta, \quad \eta \in \mathbb{C}^{N\times n}.
\end{align}
In particular, $\mathbb{V}_{p}^{\mu}$ exhibits $(p-1)$-growth, in the sense that $\mathbb{V}_{p}^{\mu}(\eta) \approx |\eta|^{p-1}$, for large $|\eta|$. Here, the notation $\mathbb{C}^{N\times n}$ stands for the spaces of $N\times n$ real-valued matrices, and in the problem settings, we also introduce the \emph{degeneracy parameter} $0 \le \mu \le 1$, which distinguishes between the degenerate ($\mu=0$) and non-degenerate case $\mu>0$. In particular, when $\mu=0$ and $\mathfrak{a}\equiv 1$, one has in mind the $p$-Laplacian system as a model case. On the right-hand side, we shall assume that the vector field $\mathbf{F} \in L^p(\Omega,\mathbb{C}^{N\times n})$. Equivalently, one could consider the right-hand side of the type $\mathrm{div}(\mathbf{G})$ under the assumption 
\begin{align*}
\mathbf{F} \in L^p(\Omega,\mathbb{C}^{N\times n}) \Longleftrightarrow \mathbf{G}=|\mathbf{F}|^{p-2}\mathbf{F} \in L^{p'}(\Omega,\mathbb{C}^{N\times n}).
\end{align*}
Corresponding to the standard ellipticity and growth conditions imposed on the differential operator on the left-hand side, modeled on the real case, it requires $|\mathfrak{a}(x)|$ to be uniformly bounded, and in this paper, we assume that
\begin{align}\label{elip-cond}
\gamma_1 \le |\mathfrak{a}(x)| =  \left(|\mathrm{Re}(\mathfrak{a}(x))|^2+|\mathrm{Im}(\mathfrak{a}(x))|^2\right)^{\frac{1}{2}} \le \gamma_2, \mbox{ for a.e. } x \in \Omega,
\end{align}
where $\gamma_1, \gamma_2$ are positive constants. The main focus is set on the structural conditions placed on the given data, from which the existence, uniqueness, and regularity properties of a complex-valued weak solution $\mathbf{u}: \Omega \to \mathbb{C}^{N}$ are derived. From the standard monotonicity condition for the differential operator on the left-hand side, the second assumption of the measurable function $\mathfrak{a}$ naturally emerges when estimating an expression of the form $\mathfrak{a}(x) \langle \mathbb{V}_{p}^{\mu}(\eta_1) - \mathbb{V}_{p}^{\mu}(\eta_2), \eta_1 - \eta_2 \rangle$, for $\eta_1, \eta_2 \in \mathbb{C}^{N\times n}$. It is remarkable here that, for the real-valued case, with a slight intentional abuse of notation, we also use the symbol $\mathbb{V}_{p}^{\mu}$ as defined in~\eqref{def-Vps}, namely
\begin{align}\notag
\mathbb{V}_{p}^{\mu}: \mathbb{R}^{2N \times n} \to \mathbb{R}^{2N \times n}, \quad z \in \mathbb{R}^{2N \times n} \mapsto \mathbb{V}_{p}^{\mu}(z) = (\mu^2 + |z|^{2})^{\frac{p-2}{2}}z.
\end{align}
Recalling that in the real-valued setting, a basic property of the operator reads as: there exist two constants $c_1,c_2>0$ depending on $p,N,n$ such that
\begin{align}\label{basic-ineq}
c_1 \left(\mu^2 + |y|^2 + |z|^2\right)^{\frac{p-2}{2}}|y - z|^2 \le \left(\mathbb{V}_{p}^{\mu}(y) - \mathbb{V}_{p}^{\mu}(z)\right) : (y - z) \le c_2 \left(\mu^2 + |y|^2 + |z|^2\right)^{\frac{p-2}{2}}|y - z|^2,
\end{align}
for every $y,z \in \mathbb{R}^{2N \times n}$, where $:$ denotes the canonical inner product in $\mathbb{R}^{2N \times n}$ (see Section~\ref{sec:pre}). In view of the presence of two constants $c_1,c_2$ in~\eqref{basic-ineq}, we further assume that there exists a constant $0<\gamma_0<\infty$ such that
\begin{align}\label{new-cond}
\mathrm{Re}(\mathfrak{a}(x))-c_0|\mathrm{Im}(\mathfrak{a}(x))|>\gamma_0, \ \mbox{ for a.e. } x \in \Omega,
\end{align}
where $c_0$ is given by
\begin{align}\label{c_0}
c_0 = \sqrt{\left(\frac{c_2}{c_1}\right)^2-1}.
\end{align} 
In our context, the imaginary part of the coefficients typically gives rise to dispersive effects, which need to be controlled to preserve the accretive structure of the leading operator on the left-hand side (closely related to the dissipativity condition in the sense of Cialdea-Maz’ya in~\cite{CM2005}). The conditions~\eqref{elip-cond} and~\eqref{new-cond} tell us that the ratio of the imaginary and real parts of $\mathfrak{a}$ is uniformly bounded, i.e. $|\mathrm{Im}(\mathfrak{a}(\cdot))| \le C(\gamma_0,c_0)\mathrm{Re}(\mathfrak{a}(\cdot))$. Geometrically, the values of the coefficient $\mathfrak{a}$ are always confined to a sectorial region in the complex plane, with an angular deviation defined by the constant $C(\gamma_0,c_0)$. In addition, together with the strongly accretive ellipticity of the operator $\mathbb{V}_p^\mu$ (which will be established in Lemma~\ref{lem:Re-Im}), these assumptions guarantee the existence and uniqueness of weak solutions (followed by standard methods, such as Lax-Milgram and the monotone operator theory). We send the reader to Remark~\ref{rem:EU} for a detailed discussion. Further, it is worth noticing that even for the case of real-valued coefficients, that is, when $\mathfrak{a}: \Omega \to \mathbb{R}$, the main results of this paper are still valid under these two assumptions~\eqref{elip-cond} and~\eqref{new-cond}. 

When the coefficients are real-valued, the problem~\eqref{eq-main} is very well understood from the analytical point of view, as it falls within the broader class of (possibly degenerate) elliptic operators in divergence form, including the $p$-Laplace equations and systems. Over the past few years, there has been tremendous interest in developing Calder\'on-Zygmund theory to establish gradient estimates ($L^p$ or $W^{1,p}$ regularity) for weak solutions with real coefficients. For a selection of the extensive literature along this research line, to which the reader may consult, see, for instance,~\cite{AM2002, AM2007, BDGN2022, BW2004, DiBenedetto1993, DM2010, DM2023, Iwaniec83, KZ, PNmix, NT25, PNJDE, TN23, PNJFA}. Further, an account of recent advances, some challenging problems, and an extensive list of references can be found in the existing surveys in~\cite{Min2017, MP2020}. 

One of the distinguishing features of this paper is that the complex-valued coefficients are described by the map $x \in \Omega \mapsto \mathfrak{a}(x) \in \mathbb{C}$. The study of problems with complex-valued coefficients is motivated by geometric and physical problems, where the linear (or non-linear) equations arise in a natural way.  It has attracted increasing attention, notably in applications such as electromagnetic wave propagation in lossy dielectrics (where conductivity/permittivity can be complex), anisotropic rheology (involving complex varying coefficients), nonlinear viscoelasticity, Calder\'on problems arising in electrical impedance tomography to nonlinear materials, and the theory of $p$-harmonic complex exponentials, among others. We refer the reader to vibrant studies in~\cite{SS2010, Spence1973, Calderon1980, Uhlmann2009, Barton} and the references therein. As far as we are concerned, the existence, uniqueness, and regularity theory of elliptic problems involving complex coefficients poses several challenges requiring specialized techniques. In general, weak solutions to the problem with a measurable complex-valued coefficient $\mathfrak{a}$ can be unbounded or fail to be continuous even when $\mathfrak{a}$ is a complex constant matrix (cf.~\cite{FM2008, MNP1989, DP2020}). However, classical regularity results (such as local boundedness and continuity) can be recovered under certain restrictive conditions on the real part $\mathrm{Re}(\mathfrak{a})$, to satisfy uniform ellipticity, and additional assumptions to control the complex part $\mathrm{Im}(\mathfrak{a})$. Due to the lack of continuity and the failure of the classical maximum principle, problems involving complex coefficients are, in general, less well understood. Nonetheless, a few special cases have been previously investigated. Regarding the linear case $p=2$, the solvability and regularity of divergence form complex coefficients equations (with elliptic operator $\mathcal{L}\mathbf{u}=-\mathrm{div}(\mathfrak{a}(x)\mathrm{D}\mathbf{u})$) have been widely investigated by many authors. An interesting question, originating from the work of Maz'ya \emph{et al.} in~\cite{MNP1989} on the H\"older continuity of complex-valued solutions, has motivated systematic and extensive investigations into the regularity properties under minimal data assumptions. A non-exhaustive sample of contributions along this direction of research includes~\cite{HKMP2015, DHM2021, DP2020, A3HK2011, CDS2011, Frehse2008, FM2008}. 

One of the recent advances in the investigation of $p$-Laplacian type elliptic systems with complex-valued coefficients is a recent paper by Kim-Vestberg~\cite{KV25}, in which it has been investigated the existence, uniqueness, and global H\"older continuity of weak solutions in a rectangular-type domain. As far as we are aware, in the literature, there has been too little written on the solutions to the problem of type~\eqref{eq-main}. Existence, uniqueness, and H\"older regularity estimates developed in~\cite{KV25} present a novel contribution to the topic, where the complex-valued map $\mathfrak{a}$ is decomposed into its real and imaginary parts that satisfy two conditions: there exist $0<\nu<L<\infty$ such that
\begin{align}\label{eq:cona_KV}
\begin{cases}
\mathrm{Re}(\mathfrak{a}(x)) - \displaystyle{\frac{c_2}{c_1}}|\mathrm{Im}(\mathfrak{a}(x))| >\nu; \\
\mathrm{Re}(\mathfrak{a}(x)) + |\mathrm{Im}(\mathfrak{a}(x))| \le L,
\end{cases}
\end{align}
for almost everywhere $x \in \Omega$. Let us now discuss a notable difference of our structure conditions presented in~\eqref{elip-cond}-\eqref{new-cond} and those established in~\cite{KV25}. Specifically, compared to the assumption in~\eqref{eq:cona_KV}$_1$, condition~\eqref{new-cond} provides a natural and less restrictive improvement. Exploiting these conditions, a minor modification of the proofs in our study allows us to get similar existence and uniqueness results. Once these structural conditions are at hand, this paper aims to proceed further in the investigation of regularity theory, in particular to address the Calder\'on-Zygmund estimate for weak solutions to~\eqref{eq-main}. As a consequence, the relevant gradient bounds for $p$-Laplace systems involving complex coefficients in a variety of specific function spaces will be directly derived.

Before presenting the main results in this paper, we specify our assumptions as follows. 

\textbf{(A1) Basic assumptions on complex-valued coefficient.}  In the entirety of the paper, we deal with a divergence form elliptic system of type~\eqref{eq-main}, where the map $\mathfrak{a}: \Omega \to \mathbb{C}$ is a complex coefficient satisfying both structure conditions~\eqref{elip-cond} and~\eqref{new-cond} as mentioned above. Further, in the original problem, we also consider $\mathbb{V}_{p}^{\mu}: \mathbb{C}^{N\times n} \to \mathbb{C}^{N\times n}$ with $\mu \in [0,1]$, is defined in~\eqref{def-Vps}. It is worth noting here that the parameter $\mu \in [0,1]$ is introduced to distinguish both degenerate and non-degenerate regimes. When $\mu=0$, our results correspond precisely to elliptic systems governed by the standard $p$-Laplacian $-\mathrm{div}(\mathfrak{a}(x)|\mathrm{D}\mathbf{u}|^{p-2}\mathrm{D}\mathbf{u})=-\mathrm{div}\left(|\mathbf{F}|^{p-2}\mathbf{F}\right)$. 

\textbf{(A2) Assumption on $\partial\Omega$.} In our study, global gradient bounds are offered in domains that enjoy minimal regularity assumptions on the boundary. Here, we assume that $\Omega$ is $(\epsilon_0,r_0)$-Reifenberg flatness for $r_0>0$ and an appropriate small number $\epsilon_0 \in (0,1)$. It means that our domain is flat in the following sense: at every $x_0 \in \partial \Omega$ and $0<r<(1-\epsilon_0)r_0$, it is possible to find a coordinate system $\{\xi_1, \xi_2,\cdots,\xi_n\}$ such that the new origin ${\mathbf{0}}_n$ belongs to $\Omega$ and
\begin{align}\notag 
\begin{cases} 
x_0 =\left(0,0,...,0, - \displaystyle{\frac{r\epsilon_0}{1-\epsilon_0}}\right), \\
B_r({\mathbf{0}}_n) \cap \{\xi_n > 0\} \subset B_r({\mathbf{0}}_n) \cap \Omega \subset B_r({\mathbf{0}}_n) \cap \left\{\xi_n >  -\displaystyle{\frac{2 r\epsilon_0}{1-\epsilon_0}}\right\},\end{cases}
\end{align}
where, we simply write $\{\xi_n > \tau\}$ instead of $\{\xi = (\xi_1, \xi_2,\cdots, \xi_n): \ \xi_n > \tau\}$ for ease of notation.

To the best of our knowledge, the Reifenberg flat domain is a natural generalization of a Lipschitz domain with a small Lipschitz constant. Its boundary is non-smooth, but flat in Reifenberg's sense: it can be well-trapped between two hyperplanes, depending on the chosen scaling parameter (small enough). Some known examples of Reifenberg flat domains are $C^1$ domains, $\epsilon$-Lipschitz domains with sufficiently small $\epsilon>0$, domains with fractal boundaries, etc. Byun-Wang in~\cite{BW2004, BW2008} first employed the Reifenberg flatness to obtain the up-to-boundary regularity of elliptic equations, and this approach inspired several subsequent developments in global regularity properties of solutions to boundary value problems. This condition is often referred to as the \emph{minimal geometric requirement} on the domain to guarantee significant mathematical properties to hold, such as topological manifold structure and various analytical results, despite the non-smoothness of the boundary in nature. Further, as observed in~\cite{MT2010, LMS2014}, the Reifenberg flat domain is only applicable when the scaling parameter $\epsilon_0$ is less than a certain threshold (typically $1/8$ or a small constant depending on the dimension $n$). In the literature, one often assumes $\epsilon_0 \in (0,1/8)$, where the specific constant $1/8$ is derived from technical proofs in geometric measure theory ensuring basic topological regularity of the domain.

\textbf{(A3) Bounded mean oscillation (BMO) assumption on $\mathfrak{a}$.} To obtain higher regularity, we further impose an additional structural assumption on the coefficient $\mathfrak{a}$. Namely, we require $\mathfrak{a}$ to have a sufficiently small bounded mean oscillation, which means that it belongs to BMO space with a small BMO semi-norm. It is worth noting that this requirement is strictly weaker than the vanishing mean oscillation (VMO) condition, since the class of functions with small BMO semi-norms contains VMO as a proper subspace. To be more precise, we shall assume that the BMO semi-norm of $\mathfrak{a}$ is small, i.e.
\begin{align}\label{eq:BMO}
[\mathfrak{a}]^{r_0}_{\mathrm{BMO}} \le \delta_{\mathrm{BMO}},
\end{align}
for some constant $\delta_{\mathrm{BMO}}>0$ depending on $N, n, \gamma_1, \gamma_2, p, c_1, c_2$. Here, the BMO semi-norm of $\mathfrak{a}$ is defined by
\begin{align}\notag
[\mathfrak{a}]^{r_0}_{\mathrm{BMO}} := \sup_{y \in \mathbb{R}^n,\, 0 < \xi \le r_0} \left( \fint_{B_\xi(y)} |\mathfrak{a}(x) - \overline{\mathfrak{a}}_{B_\xi(y)}| \, dx\right),
\end{align}
for some $r_0>0$, where $\overline{\mathfrak{a}}_{B_\xi(y)}$ denotes the averaged integral of $\mathfrak{a}$ over $B_{\xi}(y)$, that is
$$\overline{\mathfrak{a}}_{B_\xi(y)} = \fint_{B_{\xi}(y)} \mathfrak{a}(x) dx = \fint_{B_{\xi}(y)} \mathrm{Re}(\mathfrak{a}(x)) dx + i \fint_{B_{\xi}(y)} \mathrm{Im}(\mathfrak{a}(x)) dx.$$

\textbf{Notation.} Before going on, let us clarify a few notations and conventions that will be adopted throughout this paper. First, for the vectorial case, $\mathbf{g}: \Omega \to \mathbb{C}^N$, i.e. $\mathbf{g}=(g_1,g_2,\dots,g_N)$, the ``gradient'' of $\mathbf{g}$, often denoted by $\mathrm{D}\mathbf{g}$, represents the Jacobian matrix of $\mathbf{g}$. It is a tensor of size $N \times n$, defined as
\begin{align}\label{def-Dg}
\mathrm{D}\mathbf{g} = \left[\frac{\partial g_k}{\partial x_j}\right]_{1 \le k \le N \ \text{and} \ 1 \le j \le n},
\end{align}
where $\displaystyle{\frac{\partial g_k}{\partial x_j}}: \Omega \to \mathbb{C}$ is a scalar complex-valued function, is understood in the weak sense, defined by
\begin{align*}
\frac{\partial g_k}{\partial x_j} = \frac{\partial \mathrm{Re}(g_k)}{\partial x_j} + i \frac{\partial \mathrm{Im}(g_k)}{\partial x_j}, \quad k \in \{1,2,\cdots,N\},\ j \in \{1,2,\cdots,n\}.
\end{align*}

Throughout the paper, we employ the letter $C$ to denote a general positive constant that may vary from line to line, with peculiar dependence on prescribed parameters. The dependencies will often be omitted within the series of estimates and specified thereafter. In what follows, we simply write $B_R(x_0)$ in place of an open ball in $\mathbb{R}^n$ centered at $x_0$ with radius $R>0$. When the center is clear from the context, and no ambiguity arises, we shall leave out the center and just write $B_R$. Furthermore, we use the notation $\Omega_R$ to implicitly mean the intersection $\Omega \cap B_{R}$, a portion of the ball $B_R$ in $\Omega$.  Besides, $|\mathcal{V}|$ denotes the Lebesgue measure of a measurable set $\mathcal{V} \subset \mathbb{R}^n$. In what follows, we denote by $\mathcal{M}eas(\Omega;\mathbb{K}^{N \times n})$ the space of all Lebesgue measurable functions $\mathbf{f}: \Omega \to \mathbb{K}^{N \times n}$. For the sake of brevity, given any function $\mathbf{f} \in \mathcal{M}eas(\Omega;\mathbb{K}^{N \times n})$ and a constant $\lambda>0$, we shall write the level set $\{x \in \Omega: \, |\mathbf{f}(x)|>\lambda\}$ simply by $\{|\mathbf{f}|>\lambda\}$.  Moreover, for any subset $\mathcal{V} \subset \Omega$ with finite positive measure and any integrable map $\phi: \mathcal{V} \to \mathbb{K}$, the integral mean of $\phi$ over $\mathcal{V}$ will be denoted by
\begin{align*}
\overline{\phi}_{\mathcal{V}} := \fint_{\mathcal{V}} \phi(x) dx = \frac{1}{|\mathcal{V}|} \int_{\mathcal{V}} \phi(x) dx.
\end{align*}
Also note here that if $\phi: \mathcal{V} \to \mathbb{C}$ is an integrable function, its integral over $\mathcal{V}$ is defined component-wise by
\begin{align*}
\int_{\mathcal{V}} \phi(x) dx = \int_{\mathcal{V}} \mathrm{Re}(\phi(x)) dx + i \int_{\mathcal{V}} \mathrm{Im}(\phi(x)) dx.
\end{align*}
Moreover, for notational simplicity, the \emph{structural data} from the initial setting of the problem will be collected in a set of relevant parameters, and deliberately not repeated in every statement, defining the set as follows
\begin{align*}
\texttt{SD} := \{n,N,p,\mu,c_1,c_2,\gamma_0,\gamma_1,\gamma_2, r_0/\mathrm{diam}(\Omega)\}.
\end{align*}

Under the assumptions aforementioned, let us start with the definition of weak solution to~\eqref{eq-main}. It is worth mentioning here that the assumption (A1) we imposed on the structural data is natural and sufficient to ensure the existence and uniqueness of weak solutions (stated in Theorem~\ref{theo:EU}). Of independent interest, global regularity estimates for the weak solutions in target function spaces will be deduced by employing the additional assumptions (A2) and (A3), see Theorem~\ref{theo-CZ} and~\ref{theo-MorreyM}.

\begin{definition}[Weak solution]
We say that a function $\mathbf{u} \in W_0^{1,p}(\Omega;\mathbb{C}^{N})$ is a weak  solution to~\eqref{eq-main} under assumption (A1) if the following variational formulation
\begin{align}\label{weak-solution}
\int_{\Omega} \mathfrak{a}(x) \langle \mathbb{V}_{p}^{\mu}(\mathrm{D}\mathbf{u}), \mathrm{D}\varphi \rangle dx
= \int_{\Omega} \langle |\mathbf{F}|^{p-2}\mathbf{F} , \mathrm{D}\varphi \rangle dx,
\end{align}
is valid for every test function $\varphi \in W^{1,p}_0(\Omega;\mathbb{C}^{N})$.
\end{definition}
Next, we shall introduce the definition of fractional maximal operators and their truncated versions, which play a significant role in the remainder of this work.

\begin{definition}[Fractional maximal operators]
\label{def:M_beta}
For a given $\beta \in [0, n)$, the maximal fractional operator, usually written by $\mathbf{M}_\beta$, is a measurable map defined on $L^1_{\mathrm{loc}}(\mathbb{R}^n)$ as follows
\begin{align}\label{Malpha}
\mathbf{M}_\beta {f}(x) = \sup_{\rho>0}{\left(\rho^{\beta} \fint_{B_\rho(x)} |{f}(z)| dz\right)}, 
\end{align}
for all $x \in \mathbb{R}^n$ and ${f} \in L^1_{\mathrm{loc}}(\mathbb{R}^n)$. Associated with $\mathbf{M}_\beta$, we further introduce two truncated fractional maximal operators of order $R>0$, defined by
\begin{align*}
\mathbf{M}_\beta^R f(x) = \sup_{0<\rho<R}{\left(\rho^\beta \fint_{B_\rho(x)}{|f(z)| dz}\right)}, \quad \mathbf{T}_\beta^R f(x) = \sup_{\rho\ge R}{\left(\rho^\beta \fint_{B_\rho(x)}{|f(z)| dz}\right)}.
\end{align*}
For each measurable set $\mathcal{V} \subset \mathbb{R}^n$, the localized counterparts of the fractional maximal operators are denoted by
$$\mathbf{M}_{\beta,\mathcal{V}} f := \mathbf{M}_{\beta} (\chi_\mathcal{V} f), \ \mathbf{M}_{\beta,\mathcal{V}}^R f := \mathbf{M}_\beta^R (\chi_\mathcal{V} f),  \mbox{ and } \ \mathbf{T}_{\beta,\mathcal{V}}^R f := \mathbf{T}_\beta^R (\chi_\mathcal{V} f), $$ 
where $\chi_\mathcal{V}$ presents the indicator function of $\mathcal{V}$. When $\mathcal{V} \equiv \Omega$, we simply write $\mathbf{M}_{\beta}$ instead of $\mathbf{M}_{\beta,\Omega}$. 
\end{definition}

\textbf{Main results.} Let us now present our main results in three theorems. The first result, Theorem~\ref{theo:EU}, states the existence and uniqueness of solutions to~\eqref{eq-main} in the weak sense. Motivated by the recent results in~\cite[Theorem 5.2, and 5.3]{KV25} under the structural conditions~\eqref{eq:cona_KV}, we prefer to highlight that in our setting,  the complex-valued function $\mathfrak{a}$ is assumed to satisfy~\eqref{elip-cond} and~\eqref{new-cond}. Secondly, we aim to derive the gradient norm estimates for weak solutions. We present two types of regularity results for solutions to~\eqref{weak-solution}, via Theorem~\ref{theo-CZ} and Theorem~\ref{theo-MorreyM}, respectively. Of particular interest, these results naturally admit extensions of Theorem~\ref{theo-CZ} in different directions, leading to new gradient estimates in a variety of function spaces, although not addressed in the present paper. 

\begin{theorem}[Existence and uniqueness]
\label{theo:EU}
Let $p>1$, $n \ge 2$, $N \ge 1$ and $\mu \in [0,1]$. Suppose that the right-hand side source term $\mathbf{F} \in L^p(\Omega;\mathbb{C}^{N \times n})$ and complex-valued coefficient $\mathfrak{a}$ satisfies assumption (A1). Then, the system~\eqref{eq-main} admits a unique weak solution $\mathbf{u} \in W_0^{1,p}(\Omega;\mathbb{C}^{N})$.
\end{theorem}

The next result specifies the gradient estimates for a weak solution in the form of Calder\'on-Zygmund type. 
\begin{theorem}[Calder\'on-Zygmund estimate]
\label{theo-CZ}
Let $u \in W_0^{1,p}(\Omega;\mathbb{C}^{N})$ be a weak solution to~\eqref{eq-main} under assumption (A1). Then, for any $\beta \in[0,n)$ and $q > 1$, there exists a sufficiently small constant 
\begin{align*}
\delta_{\mathrm{BMO}}:= \delta_{\mathrm{BMO}}(q,\beta,\texttt{SD})>0
\end{align*}
such that if $\Omega$ is $(\delta_{\mathrm{BMO}},r_0)$-Reifenberg flat in the sense of assumption (A2) for some $0<r_0<\mathrm{diam}(\Omega)$, and if the function $\mathfrak{a}$ additionally satisfies~\eqref{eq:BMO} with the pair $(\delta_{\mathrm{BMO}},r_0)$ (assumption (A3)), the following quantitative estimate holds:\begin{align}\label{CZ-Mbeta-ineq}
\|\mathbf{M}_{\beta}(|\mathrm{D}\mathbf{u}|^p)\|_{L^q(\Omega;\mathbb{R})} \le C \left(\|\mathbf{M}_{\beta}(|\mathbf{F}|^p)\|_{L^q(\Omega;\mathbb{R})} + \mu^p\right).
\end{align}
Moreover, for every $q > p$, it holds
\begin{align}\label{CZ-ineq}
\|\mathrm{D}\mathbf{u}\|_{L^q(\Omega;\mathbb{C}^{N \times n})} \le C \left(\|\mathbf{F}\|_{L^q(\Omega;\mathbb{C}^{N \times n})} + \mu\right),
\end{align}
for a constant $C \equiv C(q,\texttt{SD})>0$.
\end{theorem}

Extending the $L^q$-regularity of $\mathrm{D}\mathbf{u}$, in the next theorem, we derive a global Morrey-type bound for the gradient in terms of maximal operators. Definition of Morrey spaces, denoted by $\mathcal{M}^{q,s}$, due to Mingione in~\cite{Min2010}, will be provided in Section~\ref{sec:pre}.
\begin{theorem}[Morrey regularity]
\label{theo-MorreyM}
Let $u \in W_0^{1,p}(\Omega;\mathbb{C}^{N})$ be a weak solution to~\eqref{eq-main} under assumption (A1). Then, for every $1 < q < s < \infty$, there exists a sufficiently small constant $\delta_{\mathrm{BMO}}:= \delta_{\mathrm{BMO}}(q,s,\beta,\texttt{SD})>0$ such that if $\Omega$ is $(\delta_{\mathrm{BMO}},r_0)$-Reifenberg flat in the sense of assumption (A2) for some $0<r_0<\mathrm{diam}(\Omega)$, and if the coefficient $\mathfrak{a}$ additionally satisfies~\eqref{eq:BMO} with the pair $(\delta_{\mathrm{BMO}},r_0)$ (assumption (A3)), it holds:
\begin{align}\label{Morrey-ineq}
\|\mathbf{M}(|\mathrm{D}\mathbf{u}|^p)\|_{\mathcal{M}^{q,s}(\Omega;\mathbb{R})} \le C \left(\|\mathbf{M}(|\mathbf{F}|^p)\|_{\mathcal{M}^{q,s}(\Omega;\mathbb{R})} + \mu\right).
\end{align}
Here, $C$ is a positive constant depending on $q,s$, and structural data $\texttt{SD}$.
\end{theorem}

The remaining part of the paper is organized as follows. In Section~\ref{sec:pre}, we recall some known definitions, functional setting and state a few preliminary results that will be needed throughout the paper. Next, in Section~\ref {sec:exist}, we exploit the existence and uniqueness of weak solutions in \eqref{weak-solution}. This section sets forth the proof of Theorem~\ref{theo:EU}, whereas some discussion and conclusions on the main problem are addressed. Section~\ref{sec:comp_est} is devoted to introducing homogeneous problems and establishing comparison maps between the solution $\mathbf{u}$ of the original problem and the solution $\mathbf{v}$ of the associated homogeneous ones. In this section, we also outline the main ideas and key ingredients employed in the proofs of our results, and highlight some technical difficulties. For ease of exposition, the technical proof is split into several lemmas towards proofs of the main theorems. Finally, Section~\ref{sec:main} is the most substantial, as it contains the proofs of the main theorems, along with some remarks that are concerning both the model and methodology. More precisely, we focus on the decay estimates for solutions to~\eqref{weak-solution}, which play the key role in reaching Lebesgue- and Morrey-space regularity. 

\section{Preliminaries and function spaces}
\label{sec:pre}

In this section, we first provide a concise and informative overview of the main notions and definitions over the complex field $\mathbb{C}$ that are important for our work. As usual, ones often express the generic field as $\mathbb{K}$, which is specified to be either the field of real numbers, $\mathbb{R}$, or of complex numbers, $\mathbb{C}$; and the notation $i$ represents the imaginary unit, defined as the square root of $-1$. In the sequel, when dealing with the elliptic system for vector-valued map $\mathbf{u}$, we shall adopt the standard convention of using typeface for scalar functions and boldface for vector-valued functions. 

Let $n \ge 2$ and $N \ge 1$ be two natural numbers. For any matrix $\mathbf{Z} \in \mathbb{K}^{N \times n}$, we mean that $\mathbf{Z} = (\mathbf{Z}^1, \mathbf{Z}^2, \dots, \mathbf{Z}^N)^T$, where each $\mathbf{Z}^k \in \mathbb{K}^n$ denotes the $k^{\mathrm{th}}$ row of $\mathbf{Z}$, for all $k \in \{1,2,\cdots,N\}$.  Here and in what follows, for any $z \in \mathbb{C}$, the real and imaginary part are denoted by $\mathrm{Re}(z)$ and $\mathrm{Im}(z)$, respectively, so that $z = \mathrm{Re}(z) + i\mathrm{Im}(z)$. Analogously, for any vector $\mathbf{z} = (z_1,z_2,\dots,z_n) \in \mathbb{C}^n$, we shall define its real and imaginary parts component-wise as
\begin{align*}
\mathrm{Re}(\mathbf{z}) = (\mathrm{Re}(z_1),\mathrm{Re}(z_2),\dots,\mathrm{Re}(z_n)) \mbox{ and } \mathrm{Im}(\mathbf{z}) = (\mathrm{Im}(z_1),\mathrm{Im}(z_2),\dots,\mathrm{Im}(z_n)) \in \mathbb{R}^n.
\end{align*}
In a similar fashion, for any matrix $\mathbf{Z} \in \mathbb{C}^{N \times n}$ with its rows $\mathbf{Z}^1, \mathbf{Z}^2,\dots,\mathbf{Z}^N$, the real and imaginary parts are also defined component-wise $\mathrm{Re}(\mathbf{Z})$ and $\mathrm{Im}(\mathbf{Z}) \in \mathbb{R}^{N\times n}$. Besides, to become clear in the remainder of this study, for each matrix $\mathbf{Z} \in \mathbb{C}^{N \times n}$, we introduce the following two real-valued matrices
\begin{align}\label{eq:Z}
\widehat{\mathbf{Z}} = \begin{bmatrix}
\mathrm{Re}(\mathbf{Z}) \\ \mathrm{Im}(\mathbf{Z})
\end{bmatrix}, \quad \mbox{ and } \quad  \overset{\sim}{\mathbf{Z}} = \begin{bmatrix} -\mathrm{Im}(\mathbf{Z}) \\ \mathrm{Re}(\mathbf{Z})) \end{bmatrix} \in \mathbb{R}^{2N \times n}.
\end{align}

Also, let us recall the inner product defined on $\mathbb{R}^{2N \times n}$ as follows. For any matrices $\mathbf{D} = (\mathbf{D}^1, \mathbf{D}^2, \dots, \mathbf{D}^{2N})^T$ and $\mathbf{E} = (\mathbf{E}^1, \mathbf{E}^2, \dots, \mathbf{E}^{2N})^T \in \mathbb{R}^{2N \times n}$, the canonical inner product will be defined by
\begin{align*}
\mathbf{D} : \mathbf{E} = \sum_{k=1}^{2N} \mathbf{D}^k \cdot \mathbf{E}^k,
\end{align*}
where the notation $\cdot$ on the right-hand side presents the standard inner product in $\mathbb{R}^n$. Moreover, this inner product naturally induces the associated norm 
\begin{align*}
|\mathbf{E}|_{\mathbb{R}^{2N \times n}} =\sqrt{\mathbf{E} : \mathbf{E}}= \left(\sum_{k=1}^{2N} |\mathbf{E}^k|^2\right)^{\frac{1}{2}},
\end{align*}
where $|\mathbf{E}^k|$ denotes the standard Euclidean norm of the $k^{\mathrm{th}}$ row in $\mathbb{R}^n$. 

By the above definition of inner product and~\eqref{eq:Z}, it is clear that $\widehat{\mathbf{Z}} : \overset{\sim}{\mathbf{Z}} = 0$, for every $\mathbf{Z} \in \mathbb{C}^{N \times n}$. In virtue of these notations, the complex inner product on $\mathbb{C}^{N \times n}$ can also be defined as
\begin{align}\label{inner-C}
\langle \mathbf{Y}, \mathbf{Z} \rangle = \widehat{\mathbf{Y}} : \widehat{\mathbf{Z}} + i (\widehat{\mathbf{Y}} : \overset{\sim}{\mathbf{Z}}),
\end{align}
where $:$ denotes the standard inner product in $\mathbb{R}^{2N \times n}$. And the associated norm is then given by
\begin{align*}
|\mathbf{Z}|_{\mathbb{C}^{N \times n}} = \left( |\mathrm{Re}(\mathbf{Z})|^2 + |\mathrm{Im}(\mathbf{Z})|^2 \right)^{\frac{1}{2}},
\end{align*}
Moreover, it is easy to verify that $|\mathbf{Z}|_{\mathbb{C}^{N \times n}} = |\widehat{\mathbf{Z}}|_{\mathbb{R}^{2N \times n}} = |\overset{\sim}{\mathbf{Z}}|_{\mathbb{R}^{2N \times n}}$, and when no confusion may arise, we drop the subscripts and denote all such norms simply by $|\mathbf{Z}|$.\\[5pt]

\textbf{Lebesgue space.} Given $1\le q \le \infty$, a measurable map $\mathbf{g} \in \mathcal{M}eas(\Omega;\mathbb{K}^{N \times n})$ is said to belong to the Lebesgue space $L^q(\Omega; \mathbb{K}^{N \times n})$ if and only if
\begin{align*} 
\|\mathbf{g}\|_{L^q(\Omega; \mathbb{K}^{N \times n})} := \left(\int_{\Omega} |\mathbf{g}(x)|^q dx \right)^{\frac{1}{q}} < \infty, \quad \text{when} \ q<\infty,
\end{align*}
and if
\begin{align*}
\|\mathbf{g}\|_{L^\infty(\Omega; \mathbb{K}^{N \times n})}:=\operatorname*{ess\,sup}_{x\in\Omega} |\mathbf{g}(x)| < \infty,\quad \text{when} \ q=\infty.
\end{align*}
Here, we recall that $|\mathbf{g}|$ denotes the norm in $\mathbb{K}^N$ defined above. \\[5pt]

\textbf{Sobolev spaces.} Given $q \ge 1$, the Sobolev space $W^{1,q}(\Omega;\mathbb{K}^{N})$ is defined as the subspace of $L^q(\Omega;\mathbb{K}^{N})$ made up all the functions $\mathbf{g} \in L^q(\Omega; \mathbb{K}^{N})$ such that whose distributional gradients, $\mathrm{D}\mathbf{g}$ belong to $L^q(\Omega; \mathbb{K}^{N \times n})$. This space is endowed with the following norm:
\begin{align*} 
\|\mathbf{g}\|_{W^{1,q}(\Omega; \mathbb{K}^N)} := \left( \int_\Omega {(|\mathbf{g}(x)|^q + |\mathrm{D}\mathbf{g}(x)|^q) dx} \right)^{\frac{1}{q}} < \infty, 
\end{align*}
where $|\mathrm{D}\mathbf{g}|$ denotes the aforementioned norm in $\mathbb{K}^{N \times n}$. Here, the gradient $\mathrm{D}\mathbf{g}$ is understood in the weak sense as in~\eqref{def-Dg}, which means
\begin{align*}
\mathrm{D}\mathbf{g} = \mathrm{D}(\mathrm{Re}(\mathbf{g})) + i \mathrm{D}(\mathrm{Im}(\mathbf{g})).
\end{align*} 
Moreover, we also emphasize the closure of $C_0^{\infty}(\Omega;\mathbb{K}^{N})$ in $W^{1,q}(\Omega;\mathbb{K}^{N})$ will be denoted by $W^{1,q}_{0}(\Omega;\mathbb{K}^{N})$.\\[5pt]

\textbf{Muckenhoupt class.} For every $1 < s < \infty$, we denote by $\mathcal{A}_s$, which is the set of all functions $\omega \in L^1_{\mathrm{loc}}(\mathbb{R}^n;\mathbb{R}^+)$ such that
\begin{align*}
[\omega]_{\mathcal{A}_{s}} := \displaystyle \sup_{B \subset \mathbb{R}^n} \left(\fint_{B} \omega(x) dx\right) \left(\fint_{B} \omega(x)^{-\frac{1}{s-1}} dx\right)^{s-1} < \infty.
\end{align*}
For the case $s = 1$, we say that $\omega \in \mathcal{A}_{1}$ if there exists a constant $C>0$ such that
\begin{align*}
\mathbf{M}\omega(x)= \sup_{\varrho>0} \fint_{B_{\varrho}(x)} \omega(\zeta) d\zeta \le C \omega(x), \quad \text{for a.e. }\  x \in \mathbb{R}^n,
\end{align*}
where $\mathbf{M}$ stands for the Hardy-Littlewood maximal operator. In this case, $[\omega]_{\mathcal{A}_{1}}$ is defined as the infimum of all such values $C$ for which the above inequality holds. Moreover, we define the Muckenhoupt class $\mathcal{A}_\infty$ as the union of $\mathcal{A}_s$ classes for all $s \ge 1$, namely,
$$\mathcal{A}_{\infty} := {\bigcup_{s \ge 1} \mathcal{A}_{s}}.$$  

From the above definitions, one has the inclusion  $\mathcal{A}_{1} \subset \mathcal{A}_{s} \subset \mathcal{A}_{\infty}$, for any $1<s<\infty$. We also recall the known power-type doubling property as follows: if $\omega \in \mathcal{A}_{\infty}$, then there exist constants $C>0$ and $\nu>0$ depending only on $n$ and $[\omega]_{\mathcal{A}_\infty}$ such that
\begin{align}\label{Muck-w}
\int_{V}{\omega(x)dx} \le C \left(\frac{|V|}{|B|}\right)^{\nu} \int_{B}{\omega(x)dx}, 
\end{align}
for every ball $B \subset \mathbb{R}^n$ and any measurable subset $V \subset B$. For simplicity, we often use the notation $\omega(V):=\int_{V}\omega(x)dx$. Moreover, if $\omega \in \mathcal{A}_{\infty}$, we shall refer to the pair of constants $(C,\nu)$ from~\eqref{Muck-w} as the ``$\mathcal{A}_{\infty}$-characteristic'' associated with $\omega$, and for brevity, we denote it by $[\omega]_{\mathcal{A}_{\infty}}=(C,\nu)$.\\[5pt]

\textbf{Weighted Lebesgue spaces.} Based on the definition of Muckenhoupt weights given above, we now introduce the definition of weighted Lebesgue spaces. Let $q \in [1,\infty)$ and $\omega \in \mathcal{A}_{\infty}$. The weighted Lebesgue space $L^{q}_{\omega}(\Omega;\mathbb{K}^{N \times n})$ is defined as follows
\begin{align*}
L^{q}_{\omega}(\Omega;\mathbb{K}^{N \times n}) :=\left\{\mathbf{g} \in \mathcal{M}eas(\Omega;\mathbb{K}^{N \times n}): \ \|\mathbf{g}\|_{L^{q}_{\omega}(\Omega;\mathbb{K}^{N \times n})}<\infty\right\},
\end{align*}
where the norm $\|\mathbf{g}\|_{L^{q}_{\omega}(\Omega;\mathbb{K}^{N \times n})}$ is given by
\begin{align}\label{w-Lor-norm}
\|\mathbf{g}\|_{L^{q}_{\omega}(\Omega;\mathbb{K}^{N \times n})} :=  \left(\int_{\Omega} |\mathbf{g}(x)|^q \omega(x) dx\right)^{\frac{1}{q}}.
\end{align}
\\[5pt]
\textbf{Morrey spaces.} Let $1 \le q \le s < \infty$. The Morrey space, denoted by $\mathcal{M}^{q,s}(\Omega;\mathbb{K}^{N \times n})$, is defined as the following set
\begin{align*}
\mathcal{M}^{q,s}(\Omega;\mathbb{K}^{N \times n}) := \left\{\mathbf{g} \in L^q(\Omega;\mathbb{K}^{N \times n}): \ \|\mathbf{g}\|_{\mathcal{M}^{q,s}(\Omega;\mathbb{K}^{N \times n})}<\infty\right\},
\end{align*}
where the norm $\|\mathbf{g}\|_{\mathcal{M}^{q,s}(\Omega;\mathbb{K}^{N \times n})}$ is given by
\begin{align}\label{Morrey-norm}
\|\mathbf{g}\|_{\mathcal{M}^{q,s}(\Omega;\mathbb{K}^{N \times n})} :=  \sup_{y\in \Omega; \, 0<r<\mathrm{diam}(\Omega)} \left(\frac{1}{|B_r(y)|^{1-\frac{q}{s}}}\int_{\Omega_r(y)}|\mathbf{g}(x)|^{q}dx\right)^{\frac{1}{q}}.
\end{align}
It is also clear that when $q=s$, the Morrey space concides with the Lebesgue space, that is, $\mathcal{M}^{q,q}(\Omega; \mathbb{K}^{N \times n}) \equiv L^q(\Omega; \mathbb{K}^{N \times n})$.\\[5pt]

\textbf{Boundedness of fractional maximal operators in Lebesgue spaces.} For the convenience of the reader, we lay out a fundamental property for our later development, the boundedness property of $\mathbf{M}_{\beta}$, for given $\beta \in [0,n)$.  
\begin{proposition}
\label{lem:bound-M-beta}
Let $\beta \in \left[0,n\right)$ and $\mathcal{V}$ be a measurable subset of $\mathbb{R}^n$. Then, there exists a constant $C=C(n,\beta)>0$ such that for any $\lambda>0$ and $f \in L^{1}(\mathcal{V})$, one has
\begin{align}\label{M-beta-est-1}
\left|\left\{x \in \mathbb{R}^n: \ \mathbf{M}_{\beta,\mathcal{V}}f(x)>\lambda\right\}\right| \le C \left(\lambda^{-1}\int_{\mathcal{V}}|f(x)| dx\right)^{\frac{n}{n-\beta}}.
\end{align}
\end{proposition}
\begin{proof}
For $\mathcal{V} \subset \mathbb{R}^n$ be a measurable subset, $\beta \in \left[0,n\right), \lambda>0$, and $f \in L^{1}(\mathcal{V})$. Let us define $\phi = \lambda^{-1}\chi_{\mathcal{V}} f$, which clearly belongs to $L^1(\mathbb{R}^n)$. It is well-known that the Hardy-Littlewood maximal operator $\mathbf{M}$ is of weak type $(1,1)$, that is, there exists a constant $C=C(n)>0$ such that
\begin{align}\label{M-beta-est-2}
\sup_{\tau>0} \tau\left|\left\{x \in \mathbb{R}^n: \ \mathbf{M}\phi(x)>\tau\right\}\right| \le C \int_{\mathbb{R}^n}|\phi(x)| dx.
\end{align}
For every $x \in \mathbb{R}^n$  and $\rho>0$, one has $|B_\rho(x)| = \rho^n |B_1|$, where $B_1$ denotes the unit ball in $\mathbb{R}^n$. Recalling the definition of $\mathbf{M}_\beta$ in~\eqref{Malpha}, we can write
\begin{align*}
\mathbf{M}_{\beta}\phi(x) 
& = |B_1|^{\frac{\beta}{n}} \sup_{\rho>0} \left(\fint_{B_\rho(x)} |\phi(y)| dy\right)^{1-\frac{\beta}{n}} \left(\int_{B_\rho(x)} |\phi(y)| dy\right)^{\frac{\beta}{n}},
\end{align*}
which implies that $\mathbf{M}_{\beta}\phi(x)  \le |B_1|^{\frac{\beta}{n}} \big[\mathbf{M}\phi(x)\big]^{1-\frac{\beta}{n}} \|\phi\|_{L^1(\mathbb{R}^n)}^{\frac{\beta}{n}}$. Moreover, this estimate immediately leads to
\begin{align}\label{M-beta-est-3}
\left|\left\{x \in \mathbb{R}^n: \ \mathbf{M}_{\beta}\phi(x)>1\right\}\right| 
& \le \left|\left\{x \in \mathbb{R}^n: \  \mathbf{M}\phi(x)  > |B_1|^{\frac{-\beta}{n-\beta}} \|\phi\|_{L^1(\mathbb{R}^n)}^{\frac{-\beta}{n-\beta}}\right\}\right|.
\end{align}
Applying~\eqref{M-beta-est-2} to the right-hand side of~\eqref{M-beta-est-3}, we arrive at
\begin{align}\notag 
\left|\left\{x \in \mathbb{R}^n: \ \mathbf{M}_{\beta}\phi(x)>1\right\}\right| \le C \left(\int_{\mathbb{R}^n}|\phi(x)| dx\right)^{\frac{n}{n-\beta}},
\end{align}
which concludes~\eqref{M-beta-est-1}. The proof is complete.
\end{proof}
\\[5pt]
\textbf{A Reifenberg-type Covering Lemma.} The following density lemma is a consequence of a Calder\'on-Zygmund type covering argument developed in~\cite{CC1995, CP1998}. For a comprehensive account, the interested reader may refer to~\cite{AM2007,BP2014} and the bibliography therein. 
\begin{lemma}\label{lem:Vitali}
Suppose that $\Omega$ is a $(\delta_0, r_0)$-Reifenberg domain with $0<\delta_0 < \frac{1}{8}$ and $r_0>0$. Let $\mathcal{V}_1$ and $\mathcal{V}_2$ be two measurable subsets of $\Omega$ such that $|\mathcal{V}_1| \le \varepsilon |B_{R_0}|$ for some $R_0 \in (0, r_0]$ and a fixed $\varepsilon \in (0,1)$. Assume further that for every $x_0 \in \Omega$ and $0< \rho \le R_0$ the following condition holds
\begin{align}\label{lem:V-a2}
\Omega_\rho(x_0) \not\subset \mathcal{V}_2 \Longrightarrow |B_\rho(x_0)\cap \mathcal{V}_1| < \varepsilon |B_\rho(x_0)|.
\end{align}
Then, there exists a constant $C = C(n)>0$ such that $|\mathcal{V}_1| \leq C\varepsilon |\mathcal{V}_2|$.
\end{lemma}
\begin{proof}
Firstly, let us show that for any $x \in \mathcal{V}_1$, there exists $\rho_{x} \in (0, R_0)$ such that 
\begin{align}\label{Vita-1}
|B_{\rho_x}(x)\cap \mathcal{V}_1| = \varepsilon |B_{\rho_x}(x)|, \mbox{ and } |B_{\rho}(x)\cap \mathcal{V}_1| < \varepsilon |B_{\rho}(x)|, \mbox{ for all } \rho>\rho_x.
\end{align}
To this end, we consider the continuous function $\psi_x: (0,\infty) \to [0,1]$ defined by
\begin{align*}
\psi_{x}(\rho) = \frac{|B_{\rho}(x)\cap \mathcal{V}_1|}{|B_{\rho}(x)|}, \quad \rho >0.
\end{align*}
At this stage, it is easy to verify that $\displaystyle{\lim_{\rho\to 0^+}}\psi_{x}(\rho) = 1$ and $\psi_{x}(R_0) < \varepsilon$, since $|\mathcal{V}_1| \le \varepsilon |B_{R_0}|$. Therefore, it is possible to find $\rho_x \in (0,R_0)$ such that $\psi_{x}\left(\rho_{x}\right) = \varepsilon$ and $\psi_{x}(\rho)< \varepsilon$ for all $\rho > \rho_{x}$, which yields~\eqref{Vita-1}. 

By the virtue of Vitali covering lemma, since $\{B_{\rho_{x}}(x)\cap \mathcal{V}_1\}_{x \in \mathcal{V}_1}$ is an open cover of $\mathcal{V}_1$, there exists a countable subfamily of pairwise disjoint balls $\{B_{\rho_{x_k}}(x_k) : x_k \in \mathcal{V}_1\}_{k \in \mathbb{N}}$ such that
\begin{align}\notag
\mathcal{V}_1 = \bigcup_{k \in \mathbb{N}} \big[ B_{5\rho_{x_k}}(x_k) \cap \mathcal{V}_1\big].
\end{align}
Combining this identity with~\eqref{Vita-1}, it follows that
\begin{align}\label{Vita-2}
|\mathcal{V}_1| \le \sum_{k \in \mathbb{N}} \big| B_{5\rho_{x_k}}(x_k) \cap \mathcal{V}_1\big| < 5^n \varepsilon  \sum_{k \in \mathbb{N}} \big| B_{\rho_{x_k}}(x_k)\big| \le 15^n \varepsilon  \sum_{k \in \mathbb{N}} \big| B_{\rho_{x_k}}(x_k) \cap \Omega\big|,
\end{align}
where, the last inequality in~\eqref{Vita-2} comes from the fact that
\begin{align}\label{Vita-3}
|B_{\rho}(x)| \le 3^n |B_{\rho}(x) \cap \Omega|, \mbox{ for all } x \in \Omega \mbox{ and } \rho \in (0,r_0],
\end{align}
due to the assumption that $\Omega$ is $(\delta_0,r_0)$-Reifenberg flatness for some $\delta_0 < \frac{1}{8}$. Indeed, inequality~\eqref{Vita-3} is obviously satisfied if $B_{\rho}(x) \subset \Omega$. On the contrary, there exists $y \in \partial \Omega$ such that $|x-y| = \text{dist}(x,\partial \Omega) < \rho$. Since $\Omega$ is $(\delta_0,r_0)$-Reifenberg flat domain, it enables us to find a new coordinate system $\{\xi_1, \xi_2,\cdots,\xi_n\}$ such that the origin coincides with $x$, i.e. ${\mathbf{0}}_n \equiv x$ and
\begin{align}\notag 
B_{\rho}(x) \cap \left\{\xi_n > \delta\rho\right\} \subset B_{\rho}(x) \cap \Omega \subset B_{\rho}(x) \cap \left\{\xi_n >  -\delta\rho\right\}, \quad \mbox{ with } \delta = \frac{\delta_0}{1-\delta_0}.
\end{align}
This leads to~\eqref{Vita-3} via the following estimate
\begin{align}\notag
\frac{|B_{\rho}(x)|}{|B_{\rho}(x) \cap \Omega|} \le \frac{|B_{\rho}(x)|}{|B_{\rho}(x) \cap \left\{\xi_n > \delta\rho\right\}|} \le \frac{|B_{\rho}(x)|}{\left|B_{\frac{\rho-\delta\rho}{2}}\right|} = \left(\frac{2}{1-\delta}\right)^n < 3^n.  
\end{align}

On the other hand, thanks to the assumption~\eqref{lem:V-a2}, we arrive at
\begin{align}\label{Vita-4}
\bigcup_{k \in \mathbb{N}} \big[ B_{\rho_{x_k}}(x_k) \cap \Omega \big] \subset \mathcal{V}_2.
\end{align}
Combining~\eqref{Vita-2} with~\eqref{Vita-4} and the disjoint covering $\{B_{\rho_{x_k}}(x_k)\}_{k \in \mathbb{N}}$, it yields
\begin{align}\notag
|\mathcal{V}_1| \le 15^n \varepsilon  \left|\bigcup_{k \in \mathbb{N}} \big[ B_{\rho_{x_k}}(x_k) \cap \Omega\big] \right| \le 15^n \varepsilon |\mathcal{V}_2|.
\end{align}
The proof is now complete.
\end{proof}

\section{Existence and uniqueness}
\label{sec:exist}

This section is first devoted to establishing the strongly accretive and sectorial properties of the operator $\mathbb{V}_{p}^{\mu}$ in complex spaces under our main assumptions. These properties are key tools for our proof of Theorem \ref{theo:EU}. Although being inspired by the ideas put forward in \cite{KV25}, it requires more delicate arguments for our setting. 

The following lemma describes the standard properties of the operator $\mathbb{V}_{p}^{\mu}$ regarding the initial condition~\eqref{new-cond} of the complex coefficient $\mathfrak{a}$. 

\begin{lemma}\label{lem:Re-Im}
Let $p>1$, $\mu \in [0,1]$, and $\mathbb{V}_{p}^{\mu}$ be the operator defined in~\eqref{def-Vps}. Assume that $c_1,c_2>0$ are two positive constants characterizing the ellipticity in~\eqref{basic-ineq}. Then, for all $\eta_1, \eta_2 \in \mathbb{C}^{N \times n}$, the real and imaginary parts of $\mathbb{V}_{p}^{\mu}$ are bounded as follows
\begin{align}\label{Re}
\mathrm{Re}\left(\langle \mathbb{V}_{p}^{\mu}(\eta_1) - \mathbb{V}_{p}^{\mu}(\eta_2), \eta_1 - \eta_2 \rangle\right) \ge c_1 \left(\mu^2 + |\eta_1|^2 + |\eta_2|^2\right)^{\frac{p-2}{2}}|\eta_1 - \eta_2|^2,
\end{align}
and
\begin{align}\label{Im}
\left|\mathrm{Im}\left(\langle \mathbb{V}_{p}^{\mu}(\eta_1) - \mathbb{V}_{p}^{\mu}(\eta_2), \eta_1 - \eta_2 \rangle\right)\right| \le \sqrt{c_2^2-c_1^2} \left(\mu^2 + |\eta_1|^2 + |\eta_2|^2\right)^{\frac{p-2}{2}}|\eta_1 - \eta_2|^2,
\end{align}
\end{lemma}
\begin{proof}
First, for every $\eta_1, \eta_2 \in \mathbb{C}^{N \times n}$, our assumption on complex-valued operator $\mathbb{V}_{p}^{\mu}$ and the definition of inner product  in~\eqref{inner-C} imply that
\begin{align}\notag
I := \mathrm{Re}\left(\langle \mathbb{V}_{p}^{\mu}(\eta_1)
 - \mathbb{V}_{p}^{\mu}(\eta_2), \eta_1 - \eta_2 \rangle\right) & = \left(\mathbb{V}_{p}^{\mu}(\widehat{\eta_1}) - \mathbb{V}_{p}^{\mu}(\widehat{\eta_2})\right) : \left(\widehat{\eta_1}-\widehat{\eta_2}\right),
\end{align}
where $\widehat{\eta_1}, \widehat{\eta_2}$ are defined as in~\eqref{eq:Z} and $:$ denotes the canonical inner product on $\mathbb{R}^{2N \times n}$. \\
Owing to assumption~\eqref{basic-ineq}$_1$ and recalling that $|\widehat{\eta}| = |\eta|$, for any $\eta \in \mathbb{C}^{N \times n}$, we obtain
\begin{align*}
I \ge c_1 \left(\mu^2 + |\widehat{\eta_1}|^2 + |\widehat{\eta_2}|^2\right)^{\frac{p-2}{2}}|\widehat{\eta_1} - \widehat{\eta_2}|^2 = c_1 \left(\mu^2 + |\eta_1|^2 + |\eta_2|^2\right)^{\frac{p-2}{2}}|\eta_1 - \eta_2|^2,
\end{align*}
which yields~\eqref{Re}. On the other hand, we also have 
\begin{align}\notag
J := \mathrm{Im}\left(\langle \mathbb{V}_{p}^{\mu}(\eta_1)
 - \mathbb{V}_{p}^{\mu}(\eta_2), \eta_1 - \eta_2 \rangle\right) & = \left(\mathbb{V}_{p}^{\mu}(\widehat{\eta_1}) - \mathbb{V}_{p}^{\mu}(\widehat{\eta_2})\right) : \left(\overset{\sim}{\eta_1}-\overset{\sim}{\eta_2}\right).
\end{align}
Thanks to~\eqref{basic-ineq}$_2$, it follows that
\begin{align*}
I^2 + J^2 \le \left|\mathbb{V}_{p}^{\mu}(\widehat{\eta_1}) - \mathbb{V}_{p}^{\mu}(\widehat{\eta_2})\right|^2 \left|\widehat{\eta_1}-\widehat{\eta_2}\right|^2 \le c_2^2\left(\mu^2 + |\eta_1|^2 + |\eta_2|^2\right)^{\frac{p-2}{2}}|\eta_1 - \eta_2|^2.
\end{align*}
Combining this estimate and~\eqref{Re}, we arrive at~\eqref{Im}.
\end{proof}

\begin{remark}\label{rem:EU}
The assertion of Lemma~\ref{lem:Re-Im} implies that the complex-valued operator $\mathbb{V}_{p}^{\mu}$ is strongly accretive via its real part $\mathrm{Re}$ and satisfies a sector condition via its imaginary part $\mathrm{Im}$. Indeed, setting $z=\langle \mathbb{V}_{p}^{\mu}(\eta_1) - \mathbb{V}_{p}^{\mu}(\eta_2), \eta_1 - \eta_2 \rangle \in \mathbb{C}$, then it follows from~\eqref{Re} and~\eqref{Im} that
\begin{align*}
|\mathrm{Im}(z)| \le c_0 \mathrm{Re}(z).
\end{align*}
Geometrically, this estimate implies that the complex number $z$ lies in a sector of the complex plane with sector angle $\theta \in \left(0,\frac{\pi}{2} \right)$ such that $\tan\theta = c_0$. Therefore, the values of the operator $\mathbb{V}_{p}^{\mu}$ remain within a narrow sectorial region uniquely determined by $\theta$, ensuring that the deviation from the real axis is uniformly bounded. Hence, this property enables the complex-valued operator $\mathbb{V}_p^\mu$ to be treated using techniques analogous to those used for real-valued ones.

Let us now come to our discussion on the leading operator $\mathcal{L}_p\mathbf{u}:=-\mathrm{div}(\mathfrak{a}(\cdot)\mathbb{V}_p^\mu(\mathrm{D}\mathbf{u}))$ on the left-hand side of our main problem~\eqref{eq-main}. Combining the preceding estimates provided in~\eqref{Re} and~\eqref{Im} with the sectorial condition~\eqref{new-cond} on the coefficient $a(\cdot)$ leads  us to the following estimate
\begin{align}\label{eq:Lp}
\mathrm{Re}\langle \mathfrak{a}(x)\mathbb{V}_p^\mu(\eta), \eta \rangle \ge \gamma_0 |\mathbb{V}_p^\mu(\eta)| |\eta| \ge  C(c_0,\gamma_0) (\mu^2 + |\eta|^{2})^{\frac{p-2}{2}}|\eta|^2,
\end{align}
for almost everywhere $x \in \Omega$ and for all $\eta \in \mathbb{C}^{N \times n}$. Therefore, it is standard to see that although the operator $\mathcal{L}_p$ is not elliptic in the classical sense (as a real-valued operator), the estimate~\eqref{eq:Lp} ensures that it satisfies the strong monotonicity condition in the complex setting. This ensures that the structural conditions of the operator on the left-hand side of~\eqref{eq-main}, such as ellipticity and monotonicity, are preserved in the complex setting. And thus, the existence and uniqueness of weak solutions to our problem are guaranteed by standard monotonicity methods. 
\end{remark}

By applying the strongly accretive and sectorial properties of $\mathbb{V}_{p}^{\mu}$ in complex spaces in Lemma~\ref{lem:Re-Im}, we are now positioned to address the questions about existence and uniqueness of solutions to system \eqref{eq-main}.

\begin{proof}[Proof of Theorem~\ref{theo:EU}]
The proofs of existence and uniqueness follow an argument analogous to those of~\cite[Theorem 5.2, Theorem 5.3]{KV25}. The primary difference arises in a few technical estimates, where the assumption~\eqref{new-cond} is employed in place of~\eqref{eq:cona_KV}, which was assumed in~\cite{KV25}.
To prove the uniqueness, suppose that $\mathbf{u}$ and $\mathbf{v} \in W^{1,p}_0(\Omega;\mathbb{C}^{N})$ are both weak solutions satisfying~\eqref{weak-solution}. Subtracting the corresponding weak formulations and testing the resulting equation 
\begin{align}\label{unique-1}
\int_{\Omega} \mathfrak{a}(x) \langle \mathbb{V}_{p}^{\mu}(\mathrm{D}\mathbf{u}) - \mathbb{V}_{p}^{\mu}(\mathrm{D}\mathbf{v}), \mathrm{D}\mathbf{u} - \mathrm{D}\mathbf{v} \rangle dx = 0,
\end{align}
with the test function $\varphi= \mathbf{u}-\mathbf{v}$. In view of~\eqref{inner-C}, we can present
\begin{align*}
\langle \mathbb{V}_{p}^{\mu}(\mathrm{D}\mathbf{u}) - \mathbb{V}_{p}^{\mu}(\mathrm{D}\mathbf{v}), \mathrm{D}\mathbf{u} - \mathrm{D}\mathbf{v} \rangle & = \left(\mathbb{V}_{p}^{\mu}(\widehat{\mathrm{D}\mathbf{u}}) - \mathbb{V}_{p}^{\mu}(\widehat{\mathrm{D}\mathbf{v}})\right): \left(\widehat{\mathrm{D}\mathbf{u}} - \widehat{\mathrm{D}\mathbf{v}}\right) \\
& \qquad \qquad + i\left(\mathbb{V}_{p}^{\mu}(\widehat{\mathrm{D}\mathbf{u}}) - \mathbb{V}_{p}^{\mu}(\widehat{\mathrm{D}\mathbf{v}})\right): \left(\overset{\sim}{\mathrm{D}\mathbf{u}} - \overset{\sim}{\mathrm{D}\mathbf{v}}\right).
\end{align*}
Thus, thanks to Lemma~\ref{lem:Re-Im}, it allows us to estimate
\begin{align*}
\mathrm{Re} & \left(\mathfrak{a}(x) \langle \mathbb{V}_{p}^{\mu}(\mathrm{D}\mathbf{u}) - \mathbb{V}_{p}^{\mu}(\mathrm{D}\mathbf{v}), \mathrm{D}\mathbf{u} - \mathrm{D}\mathbf{v} \rangle\right) \\
 & \qquad = \mathrm{Re}(\mathfrak{a}(x)) \left(\mathbb{V}_{p}^{\mu}(\widehat{\mathrm{D}\mathbf{u}}) - \mathbb{V}_{p}^{\mu}(\widehat{\mathrm{D}\mathbf{v}})\right): \left(\widehat{\mathrm{D}\mathbf{u}} - \widehat{\mathrm{D}\mathbf{v}}\right) \\
& \qquad \qquad \qquad \qquad - \mathrm{Im}(\mathfrak{a}(x)) \left(\mathbb{V}_{p}^{\mu}(\widehat{\mathrm{D}\mathbf{u}}) - \mathbb{V}_{p}^{\mu}(\widehat{\mathrm{D}\mathbf{v}})\right): \left(\overset{\sim}{\mathrm{D}\mathbf{u}} - \overset{\sim}{\mathrm{D}\mathbf{v}}\right) \\
& \qquad \ge \left(c_1 \mathrm{Re}(\mathfrak{a}(x)) - \sqrt{c_2^2-c_1^2} |\mathrm{Im}(\mathfrak{a}(x))|\right) \left(\mu^2 + |\mathrm{D}\mathbf{u}|^2 + |\mathrm{D}\mathbf{v}|^2\right)^{\frac{p-2}{2}}|\mathrm{D}\mathbf{u} - \mathrm{D}\mathbf{v}|^2,
\end{align*}
for a.e. $x \in \Omega$. Under assumption~\eqref{new-cond}, it leads to
\begin{align}\label{unique-2}
\mathrm{Re}  \left(\mathfrak{a}(x) \langle \mathbb{V}_{p}^{\mu}(\mathrm{D}\mathbf{u}) - \mathbb{V}_{p}^{\mu}(\mathrm{D}\mathbf{v}), \mathrm{D}\mathbf{u} - \mathrm{D}\mathbf{v} \rangle\right)  \ge c_1\gamma_0 \left(\mu^2 + |\mathrm{D}\mathbf{u}|^2 + |\mathrm{D}\mathbf{v}|^2\right)^{\frac{p-2}{2}}|\mathrm{D}\mathbf{u} - \mathrm{D}\mathbf{v}|^2.
\end{align}
Combining~\eqref{unique-1} and~\eqref{unique-2}, one may conclude 
\begin{align*}
0 & = \mathrm{Re}\left(\int_{\Omega} \mathfrak{a}(x) \langle \mathbb{V}_{p}^{\mu}(\mathrm{D}\mathbf{u}) - \mathbb{V}_{p}^{\mu}(\mathrm{D}\mathbf{v}), \mathrm{D}\mathbf{u} - \mathrm{D}\mathbf{v} \rangle dx\right) \\
& \qquad \qquad \qquad \qquad \ge c_1\gamma_0 \int_{\Omega} \left(\mu^2 + |\mathrm{D}\mathbf{u}|^2 + |\mathrm{D}\mathbf{v}|^2\right)^{\frac{p-2}{2}}|\mathrm{D}\mathbf{u} - \mathrm{D}\mathbf{v}|^2 dx,
\end{align*}
which ensures that $\mathbf{u} = \mathbf{v}$ almost everwhere in $\Omega$. The existence of a solution is established through a similar refinement via the use of the Lax-Milgram theorem and the monotone operator theory, under the present assumption in~\eqref{new-cond}.
\end{proof}

\section{Homogeneous problems and comparison maps}
\label{sec:comp_est}

The approach to our global bounds for the gradient of weak solutions to system~\eqref{eq-main} relies on a precise level-set decay estimate, which is established through comparisons with the gradient of solutions to corresponding homogeneous systems. This section is devoted to deriving some technical results, which include basic regularity and a few preparatory results of comparison type. This comparison argument will play a crucial role in the next step regarding the subsequent analysis of level sets.

\begin{lemma}[Local comparison maps]
\label{lem-comp}
Let $\mathbf{u} \in W_0^{1,p}(\Omega;\mathbb{C}^{N})$ be a weak solution to~\eqref{eq-main} under assumption (A1). Given a point $x_0 \in \Omega$ and $0<\xi<r_0/4$, for some constant $0<r_0<\mathrm{diam}(\Omega)$. Then, there exists a sufficiently small constant $\delta_{\mathrm{BMO}} \in (0,1)$ such that if $\Omega$ is $(\delta_{\mathrm{BMO}},r_0)$-Reifenberg flat in the sense of assumption (A2), one can find a function $$\mathbf{v} \in W^{1,p}(\Omega_{2\xi}(x_0);\mathbb{C}^{N})$$
satisfying the interior gradient bound
\begin{align}\label{ReH-ineq}
\sup_{x \in \Omega_{\xi}(x_0)} |\mathrm{D}\mathbf{v}(x)| \le C \left(\fint_{\Omega_{2\xi}(x_0)}{(|\mathrm{D}\mathbf{v}|^p+ \mu^p) dx}\right)^{\frac{1}{p}}.
\end{align}
Moreover, there exist a constant $\vartheta \in (0,1)$ such that the following comparison estimate holds
\begin{align}\label{Comp-ineq}
\fint_{\Omega_{2\xi}(x_0)}|\mathrm{D}\mathbf{u} - \mathrm{D}\mathbf{v}|^pdx \le C\left[\left(\delta + \left([\mathfrak{a}]^{r_0}_{\mathrm{BMO}}\right)^{\vartheta}\right) \fint_{\Omega_{4\xi}(x_0)}|\mathrm{D}\mathbf{u}|^pdx + C_{\delta} \left(\fint_{\Omega_{4\xi}(x_0)}{(|\mathbf{F}|^p+ \mu^p) dx}\right)\right],
\end{align}
for every $\delta \in (0,\delta_{\mathrm{BMO}})$. Here, $C_{\delta}$ is a constant depending only on $\delta$ and structural data $\texttt{SD}$.
\end{lemma}
\begin{proof}
In the first step, we establish the interior integral comparison estimates in $B_{4\xi}(x_0) \subset \Omega$. Let us consider $\mathbf{w} \in \mathbf{u} + W^{1,p}_0(B_{4\xi}(x_0);\mathbb{C}^{N})$ be the unique solution to the following equation
\begin{align}\label{3.3}
\begin{cases}
\mathrm{div}\left(\mathfrak{a}(x) \mathbb{V}_{p}^{\mu}(\mathrm{D}\mathbf{w})\right) = 0, & \text{in}\ B_{4\xi}(x_0), \\[6pt]
\mathbf{w} = \mathbf{u}, & \text{on}\ \partial B_{4\xi}(x_0).
\end{cases}
\end{align}
Let $\mathbf{w}_{\mathbf{u}}$ be an extension of $\mathbf{w}$ to $\Omega$ and defined by setting 
\begin{align*}
\mathbf{w}_{\mathbf{u}}=\mathbf{w} \ \text{in} \  B_{4\xi}(x_0) \quad \text{and} \quad \mathbf{w}_{\mathbf{u}} = \mathbf{u} \ \text{in} \ \Omega \setminus B_{4\xi}(x_0).
\end{align*}
By testing the weak formulation~\eqref{weak-solution}  and~\eqref{3.3} with $\mathbf{u} - \mathbf{w}_{\mathbf{u}} \in W^{1,p}_0(B_{4\xi}(x_0); \mathbb{C}^N)$, one obtains
\begin{align}\label{3.4}
\fint_{B_{4\xi}(x_0)} \mathfrak{a}(x)\langle \mathbb{V}_{p}^{\mu}(\mathrm{D}\mathbf{u})
 - \mathbb{V}_{p}^{\mu}(\mathrm{D}\mathbf{w}), \mathrm{D}\mathbf{u} - \mathrm{D}\mathbf{w} \rangle\, dx  = \fint_{B_{4\xi}(x_0)} \langle |\mathbf{F}|^{p-2}\mathbf{F} , \mathrm{D}\mathbf{u}-\mathrm{D}\mathbf{w} \rangle \, dx.
\end{align}
The modulus of the left-hand side of this equation can be estimated as follows
\begin{align}\label{3.5a}
L & := \left|	\fint_{B_{4\xi}(x_0)} \mathfrak{a}(x)\langle \mathbb{V}_{p}^{\mu}(\mathrm{D}\mathbf{u})
 - \mathbb{V}_{p}^{\mu}(\mathrm{D}\mathbf{w}), \mathrm{D}\mathbf{u} - \mathrm{D}\mathbf{w} \rangle\, dx\right|\notag\\
&\ge \fint_{B_{4\xi}(x_0)} \mathrm{Re} \left[ \mathfrak{a}(x)\langle \mathbb{V}_{p}^{\mu}(\mathrm{D}\mathbf{u}) - \mathbb{V}_{p}^{\mu}(\mathrm{D}\mathbf{w}), \mathrm{D}\mathbf{u} - \mathrm{D}\mathbf{w} \rangle \right]\, dx\notag\\
&= \fint_{B_{4\xi}(x_0)} \mathrm{Re}(\mathfrak{a}(x))\mathrm{Re}\left(\langle \mathbb{V}_{p}^{\mu}(\mathrm{D}\mathbf{u}) - \mathbb{V}_{p}^{\mu}(\mathrm{D}\mathbf{w}), \mathrm{D}\mathbf{u} - \mathrm{D}\mathbf{w} \rangle\right) dx\notag\\
&\qquad \qquad -\fint_{B_{4\xi}(x_0)} \mathrm{Im}(\mathfrak{a}(x))\mathrm{Im}\left(\langle \mathbb{V}_{p}^{\mu}(\mathrm{D}\mathbf{u}) - \mathbb{V}_{p}^{\mu}(\mathrm{D}\mathbf{w}), \mathrm{D}\mathbf{u} - \mathrm{D}\mathbf{w} \rangle\right) dx \notag \\
&\ge \fint_{B_{4\xi}(x_0)} \mathrm{Re}(\mathfrak{a}(x))\mathrm{Re}\left(\langle \mathbb{V}_{p}^{\mu}(\mathrm{D}\mathbf{u}) - \mathbb{V}_{p}^{\mu}(\mathrm{D}\mathbf{w}), \mathrm{D}\mathbf{u} - \mathrm{D}\mathbf{w} \rangle\right) dx\notag\\
&\qquad \qquad -\fint_{B_{4\xi}(x_0)} \left|\mathrm{Im}(\mathfrak{a}(x))\right| \left|\mathrm{Im}\left(\langle \mathbb{V}_{p}^{\mu}(\mathrm{D}\mathbf{u}) - \mathbb{V}_{p}^{\mu}(\mathrm{D}\mathbf{w}), \mathrm{D}\mathbf{u} - \mathrm{D}\mathbf{w} \rangle\right)\right| dx.
\end{align}
We note that the first inequality in~\eqref{3.5a} is not sharp, as it disregards the imaginary part of the integral. Nevertheless, this estimate is essentially optimal, as the neglected term typically vanishes in general. By virtue of Lemma~\ref{lem:Re-Im}, it follows from~\eqref{3.5a} that
\begin{align}
L &\ge \fint_{B_{4\xi}(x_0)} \left(c_1 \mathrm{Re}(\mathfrak{a}(x)) - \sqrt{c_2^2-c_1^2}|\mathrm{Im}(\mathfrak{a}(x))|\right)(\mu^2 + |\mathrm{D}\mathbf{u}|^2 + |\mathrm{D}\mathbf{w}|^2)^{\frac{p-2}{2}}
|\mathrm{D}\mathbf{u} - \mathrm{D}\mathbf{w}|^2 dx.\notag
\end{align}
Assumption~\eqref{new-cond} ensures that
\begin{align}\label{3.5}
L &\ge c_1\gamma_0	\fint_{B_{4\xi}(x_0)} (\mu^2 + |\mathrm{D}\mathbf{u}|^2 + |\mathrm{D}\mathbf{w}|^2)^{\frac{p-2}{2}} |\mathrm{D}\mathbf{u} - \mathrm{D}\mathbf{w}|^2 dx.
\end{align}
On the other hand, by taking the modulus of the right-hand side, one has
\begin{align*}
\left|	\fint_{B_{4\xi}(x_0)} \langle |\mathbf{F}|^{p-2}\mathbf{F} , \mathrm{D}\mathbf{u}-\mathrm{D}\mathbf{w} \rangle  dx\right|\le\fint_{B_{4\xi}(x_0)}|\mathbf{F}|^{p-1}|\mathrm{D}\mathbf{u}-\mathrm{D}\mathbf{w}| dx,
\end{align*}
and by combining with~\eqref{3.4} and~\eqref{3.5}, we infer that
\begin{align}\label{3.8a}
\fint_{B_{4\xi}(x_0)} (\mu^2 + |\mathrm{D}\mathbf{u}|^2 + |\mathrm{D}\mathbf{w}|^2)^{\frac{p-2}{2}}
|\mathrm{D}\mathbf{u} - \mathrm{D}\mathbf{w}|^2 dx \le \gamma_0^{-1}\fint_{B_{4\xi}(x_0)}|\mathbf{F}|^{p-1}|\mathrm{D}\mathbf{u}-\mathrm{D}\mathbf{w}| dx.	
\end{align}
On making use of this inequality and Young's inequality, it follows that
\begin{align}\label{3.4a}
\fint_{B_{4\xi}(x_0)}|\mathrm{D}\mathbf{u} - \mathrm{D}\mathbf{w}|^pdx \le \delta \fint_{B_{4\xi}(x_0)}|\mathrm{D}\mathbf{u}|^pdx + C_{\delta} \left(\fint_{B_{4\xi}(x_0)}|\mathbf{F}|^p + \mu^p dx\right),
\end{align}
for each $\delta>0$. Indeed, the proof of ~\eqref{3.4a} is straightforward in the case $p \ge 2$ by using the following fundamental inequality
$$\fint_{B_{4\xi}(x_0)} |\mathrm{D}\mathbf{u} - \mathrm{D}\mathbf{w}|^p dx \le 2^{\frac{p-2}{2}} \fint_{B_{4\xi}(x_0)} (\mu^2 + |\mathrm{D}\mathbf{u}|^2 + |\mathrm{D}\mathbf{w}|^2)^{\frac{p-2}{2}} |\mathrm{D}\mathbf{u} - \mathrm{D}\mathbf{w}|^2 dx.$$
For the remaining case $1<p<2$, there exists a constant $C^*>0$ depending only on $p$ such that
\begin{align}\label{sp-ineq}
(\mu^2 + |\mathrm{D}\mathbf{u}|^2 + |\mathrm{D}\mathbf{w}|^2)^\frac{p}{2}\le \mu^p + |\mathrm{D}\mathbf{u}|^p +  |\mathrm{D}\mathbf{w}|^p\le C^*\left(\mu^p + |\mathrm{D}\mathbf{u}|^p +  |\mathrm{D}\mathbf{u}-\mathrm{D}\mathbf{w}|^p\right).
\end{align}
In the next step, we decompose $|\mathrm{D}\mathbf{u} - \mathrm{D}\mathbf{w}|^p$ as follows
\begin{align}\notag
|\mathrm{D}\mathbf{u} - \mathrm{D}\mathbf{w}|^p = (\mu^2 + |\mathrm{D}\mathbf{u}|^2 + |\mathrm{D}\mathbf{w}|^2)^{\frac{p(2-p)}{4}} \left[(\mu^2 + |\mathrm{D}\mathbf{u}|^2 + |\mathrm{D}\mathbf{w}|^2)^{\frac{p-2}{2}}
|\mathrm{D}\mathbf{u} - \mathrm{D}\mathbf{w}|^2\right]^\frac{p}{2},
\end{align}
and then applying Young’s inequality, we arrive at
\begin{align*}
\fint_{B_{4\xi}(x_0)} |\mathrm{D}\mathbf{u} - \mathrm{D}\mathbf{w}|^p dx &\le \frac{\delta}{C^*(1+2\delta)}\fint_{B_{4\xi}(x_0)}(\mu^2 + |\mathrm{D}\mathbf{u}|^2 + |\mathrm{D}\mathbf{w}|^2)^\frac{p}{2} dx \\
& \qquad \qquad  + C_\delta \fint_{B_{4\xi}(x_0)} (\mu^2 + |\mathrm{D}\mathbf{u}|^2 + |\mathrm{D}\mathbf{w}|^2)^{\frac{p-2}{2}} |\mathrm{D}\mathbf{u} - \mathrm{D}\mathbf{w}|^2 dx.	
\end{align*}
Taking~\eqref{3.8a} and~\eqref{sp-ineq} into account, one obtains that
\begin{align}
\fint_{B_{4\xi}(x_0)} |\mathrm{D}\mathbf{u} - \mathrm{D}\mathbf{w}|^p dx &\le \frac{\delta}{1+2\delta}\fint_{B_{4\xi}(x_0)}(\mu^p + |\mathrm{D}\mathbf{u}|^p + |\mathrm{D}\mathbf{u}-\mathrm{D}\mathbf{w}|^p) dx \notag \\
& \qquad \qquad + C_\delta\fint_{B_{4\xi}(x_0)}|\mathbf{F}|^{p-1}|\mathrm{D}\mathbf{u}-\mathrm{D}\mathbf{w}| dx \notag\\
& \le \frac{\delta}{1+2\delta}\fint_{B_{4\xi}(x_0)}(\mu^p + |\mathrm{D}\mathbf{u}|^p + |\mathrm{D}\mathbf{u}-\mathrm{D}\mathbf{w}|^p) dx \notag \\
& \qquad \qquad + \frac{\delta}{1+2\delta}\fint_{B_{4\xi}(x_0)}|\mathrm{D}\mathbf{u}-\mathrm{D}\mathbf{w}|^p dx + C_\delta\fint_{B_{4\xi}(x_0)}|\mathbf{F}|^p dx,	\label{u-w-100}
\end{align}
which leads to~\eqref{3.4a}. On the other hand, let us define $\widehat{\mathbf{w}} = (\mathrm{Re}(\mathbf{w}),\mathrm{Im}(\mathbf{w}))$ and consider the following two operators $\mathcal{Q}_1, \mathcal{Q}_2: \Omega \times \mathbb{R}^{2N \times n} \to \mathbb{R}^{N \times n}$, defined by 
\begin{align*}
\mathcal{Q}_1(x,\eta) = \mathrm{Re}(\mathfrak{a}(x))\left(\mu^2 + |\mathrm{D}\eta|^2\right)^{\frac{p-2}{2}}\eta^1  -\mathrm{Im}(\mathfrak{a}(x))\left(\mu^2 + |\mathrm{D}\eta|^2\right)^{\frac{p-2}{2}}\eta^2,
\end{align*}
and
\begin{align*}
\mathcal{Q}_2(x,\eta) = \mathrm{Im}(\mathfrak{a}(x))\left(\mu^2 + |\mathrm{D}\eta|^2\right)^{\frac{p-2}{2}}\eta^1 + \mathrm{Re}(\mathfrak{a}(x))\left(\mu^2 + |\mathrm{D}\eta|^2\right)^{\frac{p-2}{2}}\eta^2,
\end{align*}
for $x \in \Omega$ and $\eta = (\eta^1,\eta^2) \in \mathbb{R}^{N \times n} \times \mathbb{R}^{N \times n}$. By~\eqref{3.3} and  using the fact that $|\mathrm{D}\widehat{\mathbf{w}}| = |\mathrm{D}\mathbf{w}|$, one can verify that $\widehat{\mathbf{w}}$ satisfies the following two systems
\begin{align}\label{eq-A12}
\mathrm{div}\left(\mathcal{Q}_1(x,\mathrm{D}\widehat{\mathbf{w}})\right) = 0, \ \mbox{ and } \ 
\mathrm{div}\left(\mathcal{Q}_2(x,\mathrm{D}\widehat{\mathbf{w}})\right) = 0, \quad \text{in}\ B_{4\xi}(x_0).
\end{align}
Under assumption~\eqref{new-cond}, we can verify that at least one of the two operators, $\mathcal{Q}_1$ or $\mathcal{Q}_2$, satisfies the standard ellipticity condition for real-valued systems. Therefore, a higher integrability of $\mathrm{D}\widehat{\mathbf{w}}$ for system~\eqref{eq-A12}. At this stage, since $|\mathrm{D}\widehat{\mathbf{w}}| = |\mathrm{D}\mathbf{w}|$, this immediately implies the higher integrability of $\mathrm{D}\mathbf{w}$. That is, there exists a constant $\theta>1$ such that
\begin{align}\label{RH-w}
\left(\fint_{B_{2\xi}(x_0)}|\mathrm{D}\mathbf{w}|^{\theta p} dx\right)^{\frac{1}{\theta}} \le C \fint_{B_{4\xi}(x_0)}{(|\mathrm{D}\mathbf{w}|^p+ \mu^p) dx}.
\end{align}

For the next step of the proof, we consider $\mathbf{v} \in \mathbf{w} + W^{1,p}_0(B_{2\xi}(x_0);\mathbb{C}^{N})$ to be the unique solution to the following homogeneous system
\begin{align}\label{3.8}
\begin{cases}
\mathrm{div}\left(\mathfrak{a}_0\mathbb{V}_{p}^{\mu}(\mathrm{D}\mathbf{v})\right) = 0, & \text{in}\ B_{2\xi}(x_0), \\[6pt]
\mathbf{v} = \mathbf{w}, & \text{on}\ \partial B_{2\xi}(x_0),
\end{cases}
\end{align}
where $\mathfrak{a}_0 = \mathfrak{a}(x_0)$. Following the similar argument, by setting the real-valued vector $\widehat{\mathbf{v}} = (\mathrm{Re}(\mathbf{v}),\mathrm{Im}(\mathbf{v}))$, it follows from~\eqref{3.8} that $\widehat{\mathbf{v}}$ solves the $p$-Laplace system with constant coefficients:
\begin{align*}
\mathrm{div}\left((\mu^2 + |\mathrm{D}\widehat{\mathbf{v}}|^{2})^{\frac{p-2}{2}}\mathrm{D}\widehat{\mathbf{v}}\right) = 0, \quad \text{in}\ B_{2\xi}(x_0).
\end{align*}
By virtue of the classical regularity results by Uhlenbeck in~\cite{U77}, it follows that $\mathrm{D}\widehat{\mathbf{v}}$ is locally bounded. Namely, 
\begin{align*}
\sup_{x \in B_{\xi}(x_0)} |\mathrm{D}\widehat{\mathbf{v}}(x)| \le C \left(\fint_{B_{2\xi}(x_0)}{(|\mathrm{D}\widehat{\mathbf{v}}|^p+ \mu^p) dx}\right)^{\frac{1}{p}},
\end{align*}
which leads to~\eqref{ReH-ineq}. On the other hand, testing~\eqref{3.3} and~\eqref{3.8} by $\mathbf{v} - \mathbf{w}$, one gets that 
\begin{align}
\fint_{B_{2\xi}(x_0)} \mathfrak{a}_0 & \langle \mathbb{V}_{p}^{\mu}(\mathrm{D}\mathbf{v})
 - \mathbb{V}_{p}^{\mu}(\mathrm{D}\mathbf{w}), \mathrm{D}\mathbf{v} - \mathrm{D}\mathbf{w} \rangle\, dx \notag \\
 & \qquad = \fint_{B_{2\xi}(x_0)} \left(\mathfrak{a}(x)-\mathfrak{a}_0\right)\langle \mathbb{V}_{p}^{\mu}(\mathrm{D}\mathbf{w}), \mathrm{D}\mathbf{v} - \mathrm{D}\mathbf{w} \rangle\, dx.\label{v-w-est1}
\end{align}
Similar to the proof of~\eqref{3.5a}, we also obtain that
\begin{align}
& \left|\fint_{B_{2\xi}(x_0)} \mathfrak{a}_0  \langle \mathbb{V}_{p}^{\mu}(\mathrm{D}\mathbf{v})
 - \mathbb{V}_{p}^{\mu}(\mathrm{D}\mathbf{w}), \mathrm{D}\mathbf{v} - \mathrm{D}\mathbf{w} \rangle\, dx\right| \notag \\
 & \qquad \ge \left(c_1 \mathrm{Re}(\mathfrak{a}_0) - \sqrt{c_2^2-c_1^2}|\mathrm{Im}(\mathfrak{a}_0)| \right) \fint_{B_{2\xi}(x_0)} (\mu^2 + |\mathrm{D}\mathbf{v}|^2 + |\mathrm{D}\mathbf{w}|^2)^{\frac{p-2}{2}} |\mathrm{D}\mathbf{u} - \mathrm{D}\mathbf{w}|^2 dx.\notag
\end{align}
Remark that by the assumption~\eqref{new-cond}, which implies to $\mathrm{Re}(\mathfrak{a}_0) - c_0|\mathrm{Im}(\mathfrak{a}_0)| \ge \gamma_0$, where $c_0$ is defined as in~\eqref{c_0}. Inserting this into~\eqref{v-w-est1}, it yields
\begin{align}
c_1\gamma_0\fint_{B_{2\xi}(x_0)} & (\mu^2 + |\mathrm{D}\mathbf{v}|^2 + |\mathrm{D}\mathbf{w}|^2)^{\frac{p-2}{2}} |\mathrm{D}\mathbf{v} - \mathrm{D}\mathbf{w}|^2 dx \notag \\
& \qquad \le \fint_{B_{2\xi}(x_0)} \left|\mathfrak{a}(x)-\mathfrak{a}_0\right| \left|\mathbb{V}_{p}^{\mu}(\mathrm{D}\mathbf{w})\right| \left|\mathrm{D}\mathbf{v} - \mathrm{D}\mathbf{w}\right|\, dx \notag \\
& \qquad \le \fint_{B_{2\xi}(x_0)} \left|\mathfrak{a}(x)-\mathfrak{a}_0\right| \left(\mu^2+|\mathrm{D}\mathbf{w}|^2\right)^{\frac{p-1}{2}} \left|\mathrm{D}\mathbf{v} - \mathrm{D}\mathbf{w}\right|\, dx.\label{u-w-200}
\end{align}
in a similar fashion to the derivation of~\eqref{u-w-100}, we infer that
\begin{align}
\fint_{B_{2\xi}(x_0)} |\mathrm{D}\mathbf{v} - \mathrm{D}\mathbf{w}|^p dx &\le \frac{\epsilon}{4}\fint_{B_{2\xi}(x_0)}|\mathrm{D}\mathbf{w}|^p  dx \notag \\
& \qquad + C_{\epsilon} \fint_{B_{2\xi}(x_0)} (\mu^2 + |\mathrm{D}\mathbf{v}|^2 + |\mathrm{D}\mathbf{w}|^2)^{\frac{p-2}{2}} |\mathrm{D}\mathbf{v} - \mathrm{D}\mathbf{w}|^2 dx, \notag
\end{align}
for every $\epsilon \in (0,1)$. Merging this with~\eqref{u-w-200} and employing Young's inequality with $\epsilon$, one obtains
\begin{align}
\fint_{B_{2\xi}(x_0)} |\mathrm{D}\mathbf{v} - \mathrm{D}\mathbf{w}|^p dx &\le \frac{\epsilon}{4}\fint_{B_{2\xi}(x_0)}|\mathrm{D}\mathbf{w}|^p  dx + \frac{\epsilon}{2} \fint_{B_{2\xi}(x_0)} |\mathrm{D}\mathbf{v} - \mathrm{D}\mathbf{w}|^p dx \notag \\
& \qquad +  C_{\epsilon} \fint_{B_{2\xi}(x_0)} \left|\mathfrak{a}(x)-\mathfrak{a}_0\right|^{\frac{p}{p-1}} \left(\mu^2+|\mathrm{D}\mathbf{w}|^2\right)^{\frac{p}{2}} dx, \notag
\end{align}
which gives
\begin{align}
\fint_{B_{2\xi}(x_0)} |\mathrm{D}\mathbf{v} - \mathrm{D}\mathbf{w}|^p dx &\le \frac{\epsilon}{2}\fint_{B_{2\xi}(x_0)}|\mathrm{D}\mathbf{w}|^p  dx \notag \\
& \qquad +  C_{\epsilon} \fint_{B_{2\xi}(x_0)} \left|\mathfrak{a}(x)-\mathfrak{a}_0\right|^{\frac{p}{p-1}} \left(\mu^2+|\mathrm{D}\mathbf{w}|^2\right)^{\frac{p}{2}} dx. \label{v-w-300}
\end{align}
Applying H\"older's inequality for the last term on the right-hand side of~\eqref{v-w-300}, it yields
\begin{align}
\fint_{B_{2\xi}(x_0)} & |\mathrm{D}\mathbf{v} - \mathrm{D}\mathbf{w}|^p dx \le \frac{\epsilon}{2}\fint_{B_{2\xi}(x_0)}|\mathrm{D}\mathbf{w}|^p  dx \notag \\
& \qquad +  C_{\epsilon} \left(\fint_{B_{2\xi}(x_0)} \left|\mathfrak{a}(x)-\mathfrak{a}_0\right|^{\frac{p\theta}{(p-1)(\theta-1)}} dx \right)^{\frac{\theta-1}{\theta}} \left(\fint_{B_{2\xi}(x_0)}\left(\mu^2+|\mathrm{D}\mathbf{w}|^2\right)^{\frac{p\theta}{2}} dx\right)^{\frac{1}{\theta}}. \notag
\end{align}
Thanks to reverse H\"older inequality~\eqref{RH-w}, it follows that
\begin{align}\label{v-w-400}
\fint_{B_{2\xi}(x_0)} & |\mathrm{D}\mathbf{v} - \mathrm{D}\mathbf{w}|^p dx \le \left[\frac{\epsilon}{2} + C_{\epsilon} \left(\fint_{B_{2\xi}(x_0)} \left|\mathfrak{a}(x)-\mathfrak{a}_0\right|^{\frac{p\theta}{(p-1)(\theta-1)}} dx \right)^{\frac{\theta-1}{\theta}} \right]\left(\fint_{B_{4\xi}(x_0)}\left(\mu^p+|\mathrm{D}\mathbf{w}|^p\right) dx\right).
\end{align}
Since $C_{\epsilon}$ can be presented by $\epsilon^{-c}$ for a constant $c=c(p)>0$, it is possible to choose $\epsilon$ such that
\begin{align*}
\frac{\epsilon}{2} = C_{\epsilon} \left(\fint_{B_{2\xi}(x_0)} \left|\mathfrak{a}(x)-\mathfrak{a}_0\right|^{\frac{p\theta}{(p-1)(\theta-1)}} dx \right)^{\frac{\theta-1}{\theta}},
\end{align*}
which is equivalent to
\begin{align*}
\epsilon = C\left(\fint_{B_{2\xi}(x_0)} \left|\mathfrak{a}(x)-\mathfrak{a}_0\right|^{\frac{p\theta}{(p-1)(\theta-1)}} dx \right)^{\vartheta},
\end{align*}
for a positive constant $\vartheta = \frac{\theta-1}{\theta(1+c)}  \in (0,1)$. Substituting this into~\eqref{v-w-400}, one gets
\begin{align}\label{v-w-500}
\fint_{B_{2\xi}(x_0)} & |\mathrm{D}\mathbf{v} - \mathrm{D}\mathbf{w}|^p dx \le C\left([\mathfrak{a}]^{r_0}_{\mathrm{BMO}}\right)^{\vartheta}\fint_{B_{4\xi}(x_0)}\left(\mu^p+|\mathrm{D}\mathbf{w}|^p\right) dx.
\end{align}
Here, we remark that $\epsilon \le C\left([\mathfrak{a}]^{r_0}_{\mathrm{BMO}}\right)^{\vartheta}$ since  $\left|\mathfrak{a}(x)-\mathfrak{a}_0\right|$ is bounded. Using the following estimate
\begin{align*}
\fint_{B_{2\xi}(x_0)}|\mathrm{D}\mathbf{u} - \mathrm{D}\mathbf{v}|^p dx &\le	C\left(\fint_{B_{2\xi}(x_0)}|\mathrm{D}\mathbf{u} - \mathrm{D}\mathbf{w}|^p dx+\fint_{B_{2\xi}(x_0)}|\mathrm{D}\mathbf{v} - \mathrm{D}\mathbf{w}|^p dx\right),
\end{align*}
which allows us to conclude~\eqref{Comp-ineq} from~\eqref{3.4a} and~\eqref{v-w-500}. Finally, in the case where $x_0$ is near the boundary, i.e., $B_{4\xi}(x_0) \not\subset \Omega$, we can employ analogous estimates by applying the geometric properties of a Reifenberg flat domain, similar to the treatment of elliptic systems with real-valued coefficients. We refer the reader to works such as. We refer the reader to seminal papers such as~\cite{BW2004} or~\cite{TN23} for detailed proofs. As a consequence, we point out the existence of a constant $\delta_{\mathrm{BMO}} \in (0,1)$ such that if $\Omega$ is $(\delta_{\mathrm{BMO}},r_0)$-Reifenberg flat domain, one can find a function $\mathbf{v} \in W^{1,p}(\Omega_{2\xi}(x_0);\mathbb{C}^{N})$ satisfying both~\eqref{ReH-ineq} and~\eqref{Comp-ineq}. The proof is now complete.
\end{proof}
\\

The technique developed in this paper enables us to treat estimates over level sets, where the level sets of the gradient $\mathrm{D}\mathbf{u}$ are locally controlled by those of the structural data. The comparison strategy established in Lemma~\ref{lem-comp} leads to the following lemma, which in turn, yields the conclusion of Theorem~\ref{theo-CZ}. A distinct point emphasizing is that the level-set inequality is constructed in terms of fractional maximal distribution functions $\mathbf{M}_\beta$, for $\beta \in [0,n)$.

\begin{lemma}
\label{theo-LV}
Let $\beta \in [0,n)$, $\mathbf{F} \in L^p(\Omega;\mathbb{C}^{N \times n})$ and $\mathbf{u} \in W_0^{1,p}(\Omega;\mathbb{C}^{N})$ be a weak solution to~\eqref{eq-main} under assumption (A1). Then, for every $\varepsilon \in (0,1)$, there exist constants 
$$\sigma=\sigma(\beta,\texttt{SD})>0, \ \delta_{\mathrm{BMO}}=\delta_{\mathrm{BMO}}(\varepsilon,\beta,\texttt{SD})>0,\ \kappa=\kappa(\varepsilon,\delta_{\mathrm{BMO}},\beta,\texttt{SD})>0,$$ 
such that if,  for some $0<r_0<\mathrm{diam}(\Omega)$, $[\mathfrak{a}]^{r_0}_{\mathrm{BMO}} \le \delta_{\mathrm{BMO}}$ in view of assumption (A3) and $\Omega$ is $(\delta_{\mathrm{BMO}},r_0)$-Reifenberg flat according to assumption (A2), the following level-set estimate 
\begin{align}\label{LV-ineq}
\left|\left\{\mathbf{M}_{\beta}(|\mathrm{D}\mathbf{u}|^p) > \lambda; \ \mathbf{M}_{\beta}(|\mathbf{F}|^p) \le \kappa\lambda\right\}\right| \le C \varepsilon \left|\left\{\mathbf{M}_{\beta}(|\mathrm{D}\mathbf{u}|^p) > \sigma\lambda\right\}\right|
\end{align}
holds for all $\lambda \ge \mu^p\mathrm{diam}(\Omega)^{\beta} \kappa^{-1}$.
\end{lemma}
\begin{proof}
Let us denote two subsets being considered by $\mathcal{V}_{\lambda,\kappa}$ and $\mathcal{W}_{\lambda,\sigma}$ as follows 
\begin{align}\label{VW-set}
\mathcal{V}_{\lambda,\kappa} := \left\{\mathbf{M}_{\beta}(|\mathrm{D}\mathbf{u}|^p) > \lambda; \ \mathbf{M}_{\beta}(|\mathbf{F}|^p) \le \kappa\lambda\right\}, \ \mbox{ and } \ \mathcal{W}_{\lambda,\sigma} := \left\{\mathbf{M}_{\beta}(|\mathrm{D}\mathbf{u}|^p) > \sigma\lambda\right\}.
\end{align}
The proof of this lemma can be handled in a standard way. For the reader's convenience, we shall split it into the following steps.

\textbf{Step 1.} We first verify that $\mathcal{V}_{\lambda,\kappa}$ satisfies the first condition required by the hypothesis of Lemma~\ref{lem:Vitali} for all $\lambda>0$ large enough and $\kappa>0$ sufficiently small. More precisely, given $0<R_0\le r_0/8$, we need to show that
\begin{align}\label{ineq-step1}
|\mathcal{V}_{\lambda,\kappa}| \le \varepsilon R_0^n.
\end{align}
Recalling the definition of $\mathcal{V}_{\lambda,\kappa}$ in~\eqref{VW-set} and applying the boundedness of the fractional maximal operators from Proposition~\ref{lem:bound-M-beta}, we obtain
\begin{align}\label{3.17}
|\mathcal{V}_{\lambda,\kappa}| \le  \left|\left\{\mathbf{M}_{\beta}(|\mathrm{D}\mathbf{u}|^p) > \lambda\right\}\right|  \le C \left(\lambda^{-1}\int_\Omega|\mathrm{D}\mathbf{u}|^p dx\right)^{\frac{n}{n-\beta}} \le C \left(\lambda^{-1}\int_\Omega{(|\mathbf{F}|^p+ \mu^p) dx}\right)^{\frac{n}{n-\beta}}.
\end{align}
We remark that the last inequality in~\eqref{3.17} is a standard global a priori estimate, which can be established by a similar argument as in the proof of~\eqref{3.4a}.  Without loss of generality, we may assume that $\mathcal{V}_{\lambda,\kappa}$ is non-empty, since otherwise~\eqref{ineq-step1} trivially holds. Hence, it is possible to find $\nu_1\in\Omega$ such that $\mathbf{M}_{\beta}(|\mathbf{F}|^p)(\nu_1) \le \kappa\lambda$, which  yields
\begin{align}\notag
\rho^{\beta} \fint_{B_\rho(\nu_1)} |\mathbf{F}|^p  dx\le \kappa\lambda, \quad \mbox{ for all } \rho>0.
\end{align}
By taking $\rho = 2\mathrm{diam}(\Omega)$ and noticing the fact that $\Omega\subset B_{2\mathrm{diam}(\Omega)}(\nu_1)$, it gives us
\begin{align}\label{3.15}
\int_{\Omega} |\mathbf{F}|^p  dx \le \int_{B_{2\mathrm{diam}(\Omega)}(\nu_1)}|\mathbf{F}|^pdx \le C \mathrm{diam}(\Omega)^{n-\beta} \kappa\lambda.
\end{align} 
Substituting~\eqref{3.15} into~\eqref{3.17}, it follows that
\begin{align*}
|\mathcal{V}_{\lambda,\kappa}| \le C\left(\mathrm{diam}(\Omega)^{n-\beta}\kappa\right)^{\frac{n}{n-\beta}}= C_0\kappa^{\frac{n}{n-\beta}}R_0^n \le \varepsilon R_0^n,
\end{align*}
where $C_0=C(n,\omega,\beta,R_0/\mathrm{diam}(\Omega))$, $\lambda$ and $\kappa$ satisfy the following constraints
\begin{align}\label{lamb-kappa}
\lambda \ge \mu^p\mathrm{diam}(\Omega)^{\beta} \kappa^{-1}, \ \mbox{ and } \ 0< \kappa \le \left(C_0^{-1}\varepsilon\right)^\frac{n-\beta}{n}.
\end{align}
We now conclude that~\eqref{ineq-step1} holds under the assumption~\eqref{lamb-kappa}.\\

\textbf{Step 2.} In the next step, we shall verify that the pair of sets $(\mathcal{V}_{\lambda,\kappa},\mathcal{W}_{\lambda,\sigma})$ validates the second hypothesis of Lemma~\ref{lem:Vitali}. To this end, let us fix a $x_0\in\Omega$ and $0<R_0\le r_0/8$. For $0<\xi\le R_0$ satisfying $\Omega_\xi(x_0) \not\subset \mathcal{W}_{\lambda,\sigma}$, it is sufficient to show that 
\begin{align}\label{3.18}
|B_\xi(x_0)\cap\mathcal{V}_{\lambda,\kappa}| \le \varepsilon |B_\xi(x_0)|.	
\end{align}
The proof of~\eqref{3.18} is trivial if $B_\xi(x_0)\cap\mathcal{V}_{\lambda,\kappa}=\varnothing$. Hence, we may assume that both sets $\Omega_\xi(x_0) \setminus \mathcal{W}_{\lambda,\sigma}$ and $B_\xi(x_0)\cap\mathcal{V}_{\lambda,\kappa}$ are non-empty. Consequently, there exist $\nu_2,\nu_3\in B_\xi(x_0)$ satisfying 
\begin{align}\label{3.19a}
\mathbf{M}_{\beta}(|\mathrm{D}\mathbf{u}|^p)(\nu_2) \le \sigma\lambda, \quad \text{and} \quad \mathbf{M}_{\beta}(|\mathbf{F}|^p)(\nu_3) \le \kappa\lambda.
\end{align}
The proof of~\eqref{3.18} will be a two-step process.  Firstly, we show that the fractional maximal operator $\mathbf{M}_{\beta}$ in $\mathcal{V}_{\lambda,\kappa}$ can be replaced by its cut-off version,  $\mathbf{M}_{\beta,\Omega_{2\xi}(x_0)}^{\xi}$. To be more precise, we show that
\begin{align}\label{3.19}
\left|B_\xi(x_0)\cap\mathcal{V}_{\lambda,\kappa}\right| \le \left|\left\{\upsilon\in B_\xi(x_0): \ \mathbf{M}_{\beta,\Omega_{2\xi}(x_0)}^{\xi} \left(|\mathrm{D}\mathbf{u}|^p\right)(\upsilon)>\lambda\right\}\right|.
\end{align}
The definition of $\mathcal{V}_{\lambda,\kappa}$ yields that
\begin{align*}
|B_\xi(x_0)\cap\mathcal{V}_{\lambda,\kappa}| \le  |\{\upsilon\in B_\xi(x_0):\mathbf{M}_\beta (|\mathrm{D}\mathbf{u}|^p)(\upsilon)>\lambda\}|.
\end{align*}
From this and notice that $\mathbf{M}_\beta (|\mathrm{D}\mathbf{u}|^p)(\upsilon)=\max\{\mathbf{M}_\beta^\xi (|\mathrm{D}\mathbf{u}|^p)(\upsilon),\mathbf{T}_\beta^\xi (|\mathrm{D}\mathbf{u}|^p)(\upsilon)\}$, we infer that
\begin{align}\label{est-T}
|B_\rho(x_0)\cap\mathcal{V}_{\lambda,\kappa}| & \le |\{\upsilon\in B_\xi(x_0):\mathbf{M}_\beta^\xi (|\mathrm{D}\mathbf{u}|^p)(\upsilon)>\lambda\}| \notag \\
& \qquad \qquad + |\{\upsilon\in B_\xi(x_0):\mathbf{T}_\beta^\xi (|\mathrm{D}\mathbf{u}|^p)(\upsilon)>\lambda\}|.
\end{align}
An interesting point is that the last term in~\eqref{est-T} vanishes for all $\sigma\le3^{\beta-n}$. Indeed, for any  $\upsilon\in B_\xi(x_0)$ and $\rho\ge \xi$, the inclusion  $B_\rho(\upsilon)\subset B_{3\rho}(\nu_2)$ with  ~\eqref{3.19a} implies that
\begin{align*}
\mathbf{T}_\beta^\xi (|\mathrm{D}\mathbf{u}|^p)(\upsilon) &= \sup_{\rho\ge \xi}\rho^\beta \fint_{B_\rho(\upsilon)}|\mathrm{D}\mathbf{u}|^p dx \le 3^{n-\beta}	\mathbf{M}_\beta (|\mathrm{D}\mathbf{u}|^p)(\nu_2)\le  3^{n-\beta}\sigma\lambda \le \lambda,
\end{align*}
which ensures that $\{\upsilon\in B_\xi(x_0):\mathbf{T}_\beta^\xi (|\mathrm{D}\mathbf{u}|^p)(\upsilon)>\lambda\}=\emptyset$. On the other hand, since $B_\rho(\upsilon)\subset B_{2\xi}(x_0)$ for any $\upsilon\in B_\xi(x_0)$ and $0<\rho<\xi$, we can rewrite the cut-off maximal operator as follows
\begin{align*}
\mathbf{M}_\beta^\xi (|\mathrm{D}\mathbf{u}|^p)(\upsilon) = \sup_{0<\rho< \xi}\rho^\beta \fint_{B_\rho(\upsilon)}\chi_{\Omega_{2\xi}(x_0)}|\mathrm{D}\mathbf{u}|^p dx =\mathbf{M}_\beta^\xi \left(\chi_{\Omega_{2\xi}(x_0)}|\mathrm{D}\mathbf{u}|^p\right)(\upsilon).
\end{align*}
By gathering the two aforementioned estimates,~\eqref{3.19} follows directly from~\eqref{est-T}.\\ 

Secondly, we will apply the comparison estimates from Lemma~\ref{lem-comp}, which asserts the existence of a function $\mathbf{v} \in W^{1,p}(\Omega_{4\xi}(x_0);\mathbb{C}^{N})$ and constants $\vartheta,\theta_0,\delta_{\mathrm{BMO}} \in (0,1)$ such that 
\begin{align}\label{v-Linf}
\sup_{x \in \Omega_{2\xi}(x_0)} |\mathrm{D}\mathbf{v}(x)| \le C \left(\fint_{\Omega_{4\xi}(x_0)}{(|\mathrm{D}\mathbf{v}|^p + \mu^p) dx}\right)^{\frac{1}{p}},
\end{align}
and the following estimate
\begin{align}\label{u-v-p}
\fint_{\Omega_{4\xi}(x_0)}|\mathrm{D}\mathbf{u} - \mathrm{D}\mathbf{v}|^pdx \le C\left[\left(\delta + \left([\mathfrak{a}]^{r_0}_{\mathrm{BMO}}\right)^{\vartheta}\right) \fint_{\Omega_{8\xi}(x_0)}|\mathrm{D}\mathbf{u}|^pdx + C_{\delta} \left(\fint_{\Omega_{8\xi}(x_0)}{(|\mathbf{F}|^p+ \mu^p) dx}\right)\right],
\end{align}
holds for every $\delta \in (0,\delta_{\mathrm{BMO}})$, if $\Omega$ is a $(\delta_{\mathrm{BMO}},r_0)$-Reifenberg flat. Since $\Omega_{8\xi}(x_0) \subset \Omega_{9\xi}(\nu_2)$, the first inequality in~\eqref{3.19a} gives us
\begin{align}\label{est-u}
\fint_{\Omega_{8\xi}(x_0)}|\mathrm{D}\mathbf{u}|^pdx \le C \fint_{\Omega_{9\xi}(\nu_2)}|\mathrm{D}\mathbf{u}|^pdx \le C \xi^{-\beta}\mathbf{M}_{\beta}(|\mathrm{D}\mathbf{u}|^p)(\nu_2) \le C\xi^{-\beta}\sigma\lambda.
\end{align}
Similarly, there also holds
\begin{align}\label{est-F}
\fint_{\Omega_{8\xi}(x_0)}|\mathbf{F}|^pdx \le C \xi^{-\beta}\mathbf{M}_{\beta}(|\mathbf{F}|^p)(\nu_3) \le C\xi^{-\beta}\kappa\lambda.
\end{align}
We remind that the assumption~\eqref{lamb-kappa} ensures that $\mu^p \le \xi^{-\beta} \kappa \lambda$. Substituting this with two inequalities~\eqref{est-u}-\eqref{est-F} into~\eqref{u-v-p}, one gets that
\begin{align}\notag
\fint_{\Omega_{4\xi}(x_0)}|\mathrm{D}\mathbf{u} - \mathrm{D}\mathbf{v}|^pdx \le C \left[\left(\delta + \left([\mathfrak{a}]^{r_0}_{\mathrm{BMO}}\right)^{\vartheta}\right) \sigma + C_{\delta} \kappa\right] \xi^{-\beta}\lambda =: C \mathcal{E}(\delta,\kappa) \xi^{-\beta}\lambda,
\end{align}
where $\mathcal{E}(\delta,\kappa,\sigma)$ is given by
\begin{align}\notag
\mathcal{E}(\delta,\kappa,\sigma) := \left(\delta + \left([\mathfrak{a}]^{r_0}_{\mathrm{BMO}}\right)^{\vartheta}\right) \sigma + C_{\delta} \kappa.
\end{align}
To simplify the contribution of the parameters in $\mathcal{E}(\delta,\kappa,\sigma)$, we restrict the parameters as follows
\begin{align}\label{sig-del-kap}
\sigma<3^{-n}< \frac{1}{3}, \quad [\mathfrak{a}]^{r_0}_{\mathrm{BMO}} < \delta < 1, \quad C_{\delta} \kappa \le \frac{\delta}{3}.
\end{align}
It deduces to $\mathcal{E}(\delta,\kappa,\sigma) \le \delta^{\vartheta}$ and hence
\begin{align}\label{u-v-p-2}
\fint_{\Omega_{4\xi}(x_0)}|\mathrm{D}\mathbf{u} - \mathrm{D}\mathbf{v}|^pdx \le C \delta^{\vartheta} \xi^{-\beta}\lambda.
\end{align}
Using the elementary inequality $|\mathrm{D}\mathbf{v}|^p\le 2^{p-1}\left(|\mathrm{D}\mathbf{u}|^p + |\mathrm{D}\mathbf{u}-\mathrm{D}\mathbf{v}|^p\right)$, it deduces from~\eqref{est-u} and~\eqref{u-v-p-2} that
\begin{align*}
\fint_{\Omega_{4\xi}(x_0)}|\mathrm{D}\mathbf{v}|^p dx \le 2^{p-1} \left(\fint_{\Omega_{4\xi}(x_0)}|\mathrm{D}\mathbf{u}|^p dx + \fint_{\Omega_{4\xi}(x_0)}|\mathrm{D}\mathbf{u} - \mathrm{D}\mathbf{v}|^p dx\right)  \le C\left[\sigma + \delta^{\vartheta}\right] \xi^{-\beta}\lambda,
\end{align*}
which by~\eqref{v-Linf} guarantees 
\begin{align}\label{v-Linf-2}
\sup_{x \in \Omega_{2\xi}(x_0)} |\mathrm{D}\mathbf{v}(x)| \le C \left(\fint_{\Omega_{4\xi}(x_0)}{(|\mathrm{D}\mathbf{v}|^p + \mu^p) dx}\right)^{\frac{1}{p}} \le C\left[\sigma + \delta^{\vartheta} + \kappa\right]^{\frac{1}{p}} \left(\xi^{-\beta}\lambda\right)^{\frac{1}{p}}.
\end{align}
By the inequality $|\mathrm{D}\mathbf{u}|^p\le 2^{p-1}\left(|\mathrm{D}\mathbf{u}-\mathrm{D}\mathbf{v}|^p+|\mathrm{D}\mathbf{v}|^p\right)$, it follows that
\begin{align*}
\mathbf{M}_{\beta,\Omega_{2\xi}(x_0)}^{\xi} \left(|\mathrm{D}\mathbf{u}|^p\right)\le2^{p-1}\left[\mathbf{M}_{\beta,\Omega_{2\xi}(x_0)}^{\xi} \left(|\mathrm{D}\mathbf{u}-\mathrm{D}\mathbf{v}|^p\right) + \mathbf{M}_{\beta,\Omega_{2\xi}(x_0)}^{\xi} \left(|\mathrm{D}\mathbf{v}|^p\right)\right].
\end{align*}
Combining this inequality with~\eqref{3.19}, one obtains
\begin{align}\label{3.34}
\left|B_\xi(x_0)\cap\mathcal{V}_{\lambda,\kappa}\right| & \le \left|\left\{\upsilon\in B_\xi(x_0): \ \mathbf{M}_{\beta,\Omega_{2\xi}(x_0)}^{\xi} \left(|\mathrm{D}\mathbf{u}-\mathrm{D}\mathbf{v}|^p\right)(\upsilon)>2^{-p}\lambda\right\}\right| \notag \\
& \qquad \qquad + \left|\left\{\upsilon\in B_\xi(x_0): \ \mathbf{M}_{\beta,\Omega_{2\xi}(x_0)}^{\xi} \left(|\mathrm{D}\mathbf{v}|^p\right)(\upsilon) > 2^{-p}\lambda\right\}\right|. 
\end{align}
For any $\upsilon\in B_\xi(x_0)$ and $0<\rho < \xi$, one can check that $B_\rho(\upsilon) \subset B_{2\xi}(x_0)$. For this reason,  inequality~\eqref{v-Linf-2}  yields that
\begin{align*}
\mathbf{M}_{\beta,\Omega_{2\xi}(x_0)}^{\xi} \left(|\mathrm{D}\mathbf{v}|^p\right)(\upsilon) & = \sup_{0<\rho< \xi}\rho^\beta \fint_{B_\rho(\upsilon)}\chi_{\Omega_{2\xi}(x_0)}|\mathrm{D}\mathbf{v}|^p dx  \\
& \le \xi^\beta	\sup_{x \in \Omega_{2\xi}(x_0)} |\mathrm{D}\mathbf{v}(x)|^p \\
& \le  C_1\left[\sigma + \delta^{\vartheta} + \kappa\right] \lambda, \quad \mbox{ for all } \upsilon\in B_\xi(x_0).
\end{align*}
With this estimate, if we choose $\sigma, \delta, \kappa$ small enough such that
\begin{align}\label{sig-del-kap-2}
\sigma < 3^{-1}C_1^{-1}2^{-p}, \ \delta^{\vartheta} < 3^{-1}C_1^{-1}2^{-p}, \mbox{ and } \kappa < 3^{-1}C_1^{-1}2^{-p},
\end{align}
then the last term in~\eqref{3.34} vanishes, i.e.
\begin{align*}
\left|\left\{\upsilon\in B_\xi(x_0): \ \mathbf{M}_{\beta,\Omega_{2\xi}(x_0)}^{\xi} \left(|\mathrm{D}\mathbf{v}|^p\right)(\upsilon) > 2^{-p}\lambda\right\}\right| = 0.
\end{align*}
Therefore, inequality~\eqref{3.34} can be rewritten as
\begin{align}\notag
\left|B_\xi(x_0)\cap\mathcal{V}_{\lambda,\kappa}\right| & \le \left|\left\{\upsilon\in B_\xi(x_0): \ \mathbf{M}_{\beta,\Omega_{2\xi}(x_0)}^{\xi} \left(|\mathrm{D}\mathbf{u}-\mathrm{D}\mathbf{v}|^p\right)(\upsilon)>2^{-p}\lambda\right\}\right|. 
\end{align}
Applying the boundedness property of $\mathbf{M}_{\beta}$ and inequality~\eqref{u-v-p-2}, we arrive
\begin{align}\label{3.41}
\left|B_\xi(x_0)\cap\mathcal{V}_{\lambda,\kappa}\right| &\le C\left(2^p\lambda^{-1}\xi^n\fint_{\Omega_{2\xi}(x_0)}|\mathrm{D}\mathbf{u}-\mathrm{D}\mathbf{v}|^p dx\right)^{\frac{n}{n-\beta}} \le C_2 \delta^\frac{n\vartheta}{n-\beta}|B_\xi(x_0)|.
\end{align}
We can obtain~\eqref{3.17} from~\eqref{3.41} by choosing $\delta$ such that $C_2 \delta^\frac{n\vartheta}{n-\beta} \le \varepsilon$.\\ 

Gathering the estimates established in~\eqref{lamb-kappa}--\eqref{sig-del-kap-2}, we conclude that the pair $(\mathcal{V}_{\lambda,\kappa}, \mathcal{W}_{\lambda,\sigma})$ satisfies all the requirements of Lemma~\ref{lem:Vitali}, for all $\lambda \ge \mu^p \mathrm{diam}(\Omega)^{\beta} \kappa^{-1}$, under the following constraints
\begin{align}\notag 
\begin{cases} 0<\sigma< \min\left\{3^{-1}C_1^{-1}2^{-p}; 3^{-n}\right\}, \\
[\mathfrak{a}]^{r_0}_{\mathrm{BMO}} < \delta < \delta_{\mathrm{BMO}} < \min\left\{\left(3^{-1}C_1^{-1}2^{-p}\right)^{-\vartheta}; \left(C_2^{-1}\varepsilon\right)^\frac{n-\beta}{n\vartheta} \right\}, \\
0< \kappa \le \min\left\{3^{-1}C_1^{-1}2^{-p}; \frac{C_{\delta}^{-1}\delta}{3}; \left(C_0^{-1}\varepsilon\right)^\frac{n-\beta}{n}\right\}.\end{cases}
\end{align}
As a consequence of Lemma~\ref{lem:Vitali}, the proof of~\eqref{LV-ineq} is now complete.
\end{proof}

\section{Proofs of the main results}
\label{sec:main}

With the preceding technical lemmas at hand, in the remaining part of this paper, we focus on the proofs of the main theorems already stated in Section~\ref{sec:intro}.

\begin{proof}[Proof of Theorem~\ref{theo-CZ}]
It is straightforward to see that inequality~\eqref{CZ-Mbeta-ineq} implies~\eqref{CZ-ineq} by taking $\beta=0$ and applying the boundedness of the Hardy-Littlewood maximal operator $\mathbf{M}$ on the Lebesgue spaces $L^{q/p}(\Omega;\mathbb{R})$. Thus, it is sufficient to prove~\eqref{CZ-Mbeta-ineq}. Given $\varepsilon \in (0,1)$, thanks to Lemma~\ref{theo-LV}, there exist constants 
$$\sigma=\sigma(\beta,\texttt{SD})>0, \ \delta_{\mathrm{BMO}}=\delta_{\mathrm{BMO}}(\varepsilon,\beta,\texttt{SD})>0,\ \kappa=\kappa(\varepsilon,\delta_{\mathrm{BMO}},\beta,\texttt{SD})>0,$$ 
such that if $[\mathfrak{a}]^{r_0}_{\mathrm{BMO}} \le \delta_{\mathrm{BMO}}$ and $\Omega$ is $(\delta_{\mathrm{BMO}},r_0)$-Reifenberg flat for some $0<r_0<\mathrm{diam}(\Omega)$, then   
\begin{align}\notag
\left|\left\{\mathbf{M}_{\beta}(|\mathrm{D}\mathbf{u}|^p) > \lambda; \ \mathbf{M}_{\beta}(|\mathbf{F}|^p) \le \kappa\lambda\right\}\right| \le C \varepsilon \left|\left\{\mathbf{M}_{\beta}(|\mathrm{D}\mathbf{u}|^p) > \sigma\lambda\right\}\right|, 
\end{align}
for every $\lambda \ge \lambda_0:= \mu^p\mathrm{diam}(\Omega)^{\beta} \kappa^{-1}$. This inequality directly implies to
\begin{align}\notag
\left|\left\{\mathbf{M}_{\beta}(|\mathrm{D}\mathbf{u}|^p) > \lambda\right\}\right| \le C \varepsilon \left|\left\{\sigma^{-1}\mathbf{M}_{\beta}(|\mathrm{D}\mathbf{u}|^p) > \lambda\right\}\right| + \left|\left\{\kappa^{-1}\mathbf{M}_{\beta}(|\mathbf{F}|^p) > \lambda\right\}\right|. 
\end{align}
For every $q >1$, multiplying both sides of this inequality by $q\lambda^{q-1}$ and taking the integral over $(\lambda_0,\infty)$, one obtains
\begin{align}\notag
\int_{\lambda_0}^{\infty} q \lambda^{q-1}\left|\left\{\mathbf{M}_{\beta}(|\mathrm{D}\mathbf{u}|^p) > \lambda\right\}\right| d\lambda & \le C \varepsilon \int_{\lambda_0}^{\infty} q \lambda^{q-1}\left|\left\{\sigma^{-1}\mathbf{M}_{\beta}(|\mathrm{D}\mathbf{u}|^p) > \lambda\right\}\right| d\lambda  \\
& \qquad \qquad + \int_{\lambda_0}^{\infty} q \lambda^{q-1} \left|\left\{\kappa^{-1}\mathbf{M}_{\beta}(|\mathbf{F}|^p) > \lambda\right\}\right| d\lambda \notag \\
& \le C \varepsilon \int_{0}^{\infty} q \lambda^{q-1}\left|\left\{\sigma^{-1}\mathbf{M}_{\beta}(|\mathrm{D}\mathbf{u}|^p) > \lambda\right\}\right| d\lambda \notag  \\
& \qquad \qquad + \int_{0}^{\infty} q \lambda^{q-1} \left|\left\{\kappa^{-1}\mathbf{M}_{\beta}(|\mathbf{F}|^p) > \lambda\right\}\right| d\lambda.\label{final-1}
\end{align}
It is worth mentioning that the norm of $f$ in $L^q(\Omega;\mathbb{C}^{N \times n})$ can be presented as the integral over $(0,\infty)$ of the level sets 
\begin{align}
\int_{\Omega} |f(x)|^q dx =  \int_{0}^{\infty} q \lambda^{q-1} \left|\left\{|f|> \lambda\right\}\right| d\lambda.\label{Lp-new}
\end{align}
Using this presentation, we can decompose
\begin{align}
\|\mathbf{M}_{ \beta}(|\mathrm{D}\mathbf{u}|^p)\|_{L^q(\Omega;\mathbb{R})}^q & = \int_0^{\lambda_0} q \lambda^{q-1}\left|\left\{\mathbf{M}_{\beta}(|\mathrm{D}\mathbf{u}|^p) > \lambda\right\}\right| d\lambda + \int_{\lambda_0}^{\infty} q \lambda^{q-1}\left|\left\{\mathbf{M}_{\beta}(|\mathrm{D}\mathbf{u}|^p) > \lambda\right\}\right| d\lambda \notag \\
& \le |\Omega| \lambda_0^q + \int_{\lambda_0}^{\infty} q \lambda^{q-1}\left|\left\{\mathbf{M}_{\beta}(|\mathrm{D}\mathbf{u}|^p) > \lambda\right\}\right| d\lambda.\notag
\end{align}
Substituting~\eqref{final-1} into the last term of this inequality, we obtain
\begin{align}
\|\mathbf{M}_{ \beta}(|\mathrm{D}\mathbf{u}|^p)\|_{L^q(\Omega;\mathbb{R})}^q & \le  C\sigma^{-q}\varepsilon 	\|\mathbf{M}_{ \beta}(|\mathrm{D}\mathbf{u}|^p)\|_{L^q(\Omega;\mathbb{R})}^q + C \kappa^{-q}\left(\|\mathbf{M}_{\beta}(|\mathbf{F}|^p)\|_{L^q(\Omega;\mathbb{R})}^q + \mu^{pq}\right).\label{final-2}
\end{align}
Finally, we establish~\eqref{CZ-Mbeta-ineq} by fixing $\varepsilon = \varepsilon_0$ in~\eqref{final-2} such that $C\sigma^{-q}\varepsilon_0 < \frac{1}{2}$. At this stage, we emphasize that $\sigma$ is independent of both $\delta_{\mathrm{BMO}}$ and $\varepsilon$. The proof is complete.
\end{proof}

\begin{proof}[Proof of Theorem~\ref{theo-MorreyM}]
Following a fashion similar to the proof of Theorem~\ref{theo-CZ}, we first employ Lemma~\ref{theo-LV} to derive the level-set estimate. Specifically, for any $\varepsilon \in (0,1)$,  there exist constants 
$$\sigma=\sigma(\texttt{SD})>0, \ \delta_{\mathrm{BMO}}=\delta_{\mathrm{BMO}}(\varepsilon,\texttt{SD})>0,\ \kappa=\kappa(\varepsilon,\delta_{\mathrm{BMO}},\texttt{SD})>0,$$ 
such that if $[\mathfrak{a}]^{r_0}_{\mathrm{BMO}} \le \delta_{\mathrm{BMO}}$ and $\Omega$ is $(\delta_{\mathrm{BMO}},r_0)$-Reifenberg flat for some $0<r_0<\mathrm{diam}(\Omega)$, then   
\begin{align}\label{Morrey-est-1}
\left|\left\{\mathbf{M}(|\mathrm{D}\mathbf{u}|^p) > \lambda; \ \mathbf{M}(|\mathbf{F}|^p) \le \kappa\lambda\right\}\right| \le C \varepsilon \left|\left\{\mathbf{M}(|\mathrm{D}\mathbf{u}|^p) > \sigma\lambda\right\}\right|, 
\end{align}
for every $\lambda \ge \lambda_0:= \mu^q\mathrm{diam}(\Omega)^{\beta} \kappa^{-1}$. Let $\omega \in \mathcal{A}_{\infty}$ be a Muckenhoupt weight. By the virtue of~\eqref{Muck-w} and~\eqref{Morrey-est-1}, there exist two constant $C$ and $\nu>0$ such that
\begin{align}\notag
\omega\left(\left\{\mathbf{M}(|\mathrm{D}\mathbf{u}|^p) > \lambda; \ \mathbf{M}(|\mathbf{F}|^p) \le \kappa\lambda\right\}\right) \le C \varepsilon^{\nu} \omega\left(\left\{\mathbf{M}(|\mathrm{D}\mathbf{u}|^p) > \sigma\lambda\right\}\right).
\end{align} 
This inequality leads directly to 
\begin{align}\label{Morrey-est-2}
\omega\left(\left\{\mathbf{M}(|\mathrm{D}\mathbf{u}|^p) > \lambda\right\}\right) \le C \varepsilon^{\nu} \omega\left(\left\{\sigma^{-1}\mathbf{M}(|\mathrm{D}\mathbf{u}|^p) > \lambda\right\}\right) + \omega\left(\left\{\kappa^{-1}\mathbf{M}(|\mathbf{F}|^p) > \lambda\right\}\right). 
\end{align}
Similar to~\eqref{Lp-new}, for $q>1$ and $f \in L^{q}_{\omega}(\Omega;\mathbb{C}^{N \times n})$, it allows us to write
\begin{align}\notag
\int_{\Omega} |f(x)|^{q} \omega(x) dx =  \int_{0}^{\infty} q \lambda^{q-1} \omega\left(\left\{|f|> \lambda\right\}\right) d\lambda.
\end{align}
Applying this formula for~\eqref{Morrey-est-2}, it gives
\begin{align}\label{Morrey-est-3}
\int_{\Omega} \left[\mathbf{M}(|\mathrm{D}\mathbf{u}|^p)\right]^{q} \omega(x) dx \le  C \left( \varepsilon^{\nu} \sigma^{-1} \int_{\Omega}  \left[\mathbf{M}(|\mathrm{D}\mathbf{u}|^p)\right]^{q} \omega(x) dx + \kappa^{-1}\int_{\Omega} \left[\mathbf{M}(|\mathbf{F}|^p)\right]^{q} \omega(x) dx + \mu^q\right).
\end{align} 
By choosing $\varepsilon>0$ small enough in~\eqref{Morrey-est-3}, we arrive at
\begin{align}\label{Morrey-est-4}
\int_{\Omega} \left[\mathbf{M}(|\mathrm{D}\mathbf{u}|^p)\right]^{q} \omega(x) dx \le  C \left(\int_{\Omega} \left[\mathbf{M}(|\mathbf{F}|^p)\right]^{q} \omega(x) dx + \mu^q\right).
\end{align}
To estimate the Morrey-space norm of $\mathbf{M}(|\mathrm{D}\mathbf{u}|^p)$ in $\mathcal{M}^{q,s}(\Omega;\mathbb{R})$, where $q<s<\infty$, it is necessary to evaluate the supremum of the following integral 
\begin{align*}
\mathbb{I}_{\varrho}(y) & := \frac{1}{|B_{\varrho}(y)|^{1-\frac{q}{s}}}\int_{\Omega_{\varrho}(y)} \left[\mathbf{M}(|\mathrm{D}\mathbf{u}|^p)\right]^{q} dx,
\end{align*}
for all $y \in \Omega$ and $0<\varrho<\mathrm{diam}(\Omega)$. At this stage, it enables us to represent
\begin{align*}
\mathbb{I}_{\varrho}(y) = \frac{1}{|B_{\varrho}(y)|^{1-\frac{q}{s}}} \int_{\Omega} \left[\mathbf{M}(|\mathrm{D}\mathbf{u}|^p)\right]^{q} \chi_{B_{\varrho}(y)}(x) dx.
\end{align*}
Let us take $\tau \in (0,1)$ and set $\omega := \left(\mathbf{M} \chi_{B_{\varrho}(y)}\right)^{\tau}$. According to~\cite[Proposition 2]{CR80}, $\omega$ is a Muckenhoupt weight; and more specifically, $\omega \in \mathcal{A}_{1}$. This fact allows us to apply~\eqref{Morrey-est-4} with this choice of weight. Furthermore, we observe that $\chi_{B_{\varrho}(y)}(x) \le \left(\mathbf{M} \chi_{B_{\varrho}(y)}\right)(x) \le \omega(x) \le 1$ for almost every $x \in \mathbb{R}^n$, which leads to
\begin{align}\label{est-D-2}
\mathbb{I}_{\varrho}(y) & \le C \left( \frac{1}{|B_{\varrho}(y)|^{1-\frac{q}{s}}} \int_{\Omega} \mathcal{F}(x) \omega(x) dx + \mu^q\right),
\end{align}
where $\mathcal{F} := \left[\mathbf{M}(|\mathbf{F}|^{p})\right]^{q}$. Let us use the following decomposition $$\Omega = \Omega_{2\varrho}(y) \cup \left(\Omega \cap \bigcup_{j=1}^{\infty} B_{2^{j+1}\varrho}(y) \setminus B_{2^j\varrho}(y) \right)$$
to present the inequality in~\eqref{est-D-2} as 
\begin{align} \label{est-D-3}
\mathbb{I}_{\varrho}(y) & \le C \left( \frac{1}{|B_{\varrho}(y)|^{1-\frac{q}{s}}} \int_{\Omega_{2\varrho}(y)} \mathcal{F}(x) \omega(x) dx  + \sum_{j=1}^{\infty} \frac{1}{|B_{\varrho}(y)|^{1-\frac{q}{s}}}\int_{\Omega_{2^{j+1}\varrho}(y) \setminus B_{2^j\varrho}(y)} \mathcal{F}(x) \omega(x) dx + \mu^q\right).
\end{align}
Since $\omega(x) \le 1$ in the ball $B_{2\varrho}(y)$ and $|B_{\varrho}(y)| = {2^{-n}} |B_{2\varrho}(y)|$, we have
\begin{align}\label{Morrey-est-5}
I_0 & := \frac{1}{|B_{\varrho}(y)|^{1-\frac{q}{s}}} \int_{\Omega_{2\varrho}(y)} \mathcal{F}(x) \omega(x) dx \notag \\
& \le  2^{n\left(1-\frac{q}{s}\right)} \left(\frac{1}{|B_{2\varrho}(y)|^{1-\frac{q}{s}}} \int_{\Omega_{2\varrho}(y)} \mathcal{F}(x) dx\right) \notag \\
& \le 2^{n\left(1-\frac{q}{s}\right)} \|\mathcal{F}\|_{\mathcal{M}^{q,s}(\Omega;\mathbb{R})}^q. 
\end{align}
By direct calculation, we can show that $\omega(x) \le \frac{1}{2^{(j-1)n\tau}}$ in the annular region $B_{2^{j+1}\varrho}(y) \setminus B_{2^j\varrho}(y)$ for every $j \ge 1$. We also refer the reader to~\cite[Lemma 3.6]{NT25} for the detailed proof of this estimate. Moreover, combining this with the fact that $|B_{\varrho}(y)| = {2^{-(j+1)n}} |B_{2^{j+1}\varrho}(y)|$, one has
\begin{align}\label{Morrey-est-6}
I_j & := \frac{1}{|B_{\varrho}(y)|^{1-\frac{q}{s}}}\int_{\Omega_{2^{j+1}\varrho}(y) \setminus B_{2^j\varrho}(y)} \mathcal{F}(x) \omega(x) dx \notag \\
&\le \frac{1}{2^{(j-1)n\tau}} 2^{(j+1)n\left(1-\frac{q}{s}\right)} \left( \frac{1}{|B_{2^{j+1}\varrho}(y)|^{1-\frac{q}{s}}}\int_{\Omega_{2^{j+1}\varrho}(y)} \mathcal{F}(x) dx\right) \notag \\
& \le 2^{n\left(1-\frac{q}{s}+\tau\right)-jn\left(1-\frac{q}{s}-\tau\right)} \|\mathcal{F}\|_{\mathcal{M}^{q,s}(\Omega;\mathbb{R})}^q.
\end{align}
Substituting~\eqref{Morrey-est-5} and~\eqref{Morrey-est-6} into~\eqref{est-D-3}, it yields
\begin{align} 
\mathbb{I}_{\varrho}(y) & \le C \left( 2^{n\left(1-\frac{q}{s}\right)} \|\mathcal{F}\|_{\mathcal{M}^{q,s}(\Omega;\mathbb{R})}^q  + 2^{n\left(1-\frac{q}{s}+\tau\right)}\sum_{j=1}^{\infty} 2^{-jn\left(1-\frac{q}{s}-\tau\right)} \|\mathcal{F}\|_{\mathcal{M}^{q,s}(\Omega;\mathbb{R})}^q + \mu^q\right)\notag \\
& \le C \left(2^{n\left(1-\frac{q}{s}+\tau\right)}\sum_{j=0}^{\infty} 2^{-jn\left(1-\frac{q}{s}-\tau\right)} \|\mathcal{F}\|_{\mathcal{M}^{q,s}(\Omega;\mathbb{R})}^q + \mu^q\right).\notag
\end{align}
Finally, let us set $0< \tau < 1-\frac{q}{s}$, which ensures that 
$$\sum_{j=0}^{\infty} 2^{-jn\left(1-\frac{q}{s}-\tau\right)}< \infty.$$ 
The proof of~\eqref{Morrey-ineq} is complete by taking the supremum for all $y \in \Omega$ and $0<\varrho<\mathrm{diam}(\Omega)$.
\end{proof}


\section*{Acknowledgement}
This research is funded by Vietnam National Foundation for Science and Technology Development (NAFOSTED), Grant Number: 101.02-2025.03.

\section*{Conflict of Interest}
The authors declared that they have no conflict of interest.


\end{document}